\documentclass[a4paper,12pt]{article}

\usepackage[top=2.5cm,bottom=2.5cm,left=2.5cm,right=2.5cm]{geometry}
\usepackage{cite, amsmath, amssymb}
\usepackage[margin=1cm,
                font=small,%
                format=hang,%
                labelsep=period,%
                labelfont=bf]{caption}
\usepackage{fancyhdr}
\usepackage{pifont}
\usepackage{latexsym}
\usepackage{amsthm}
\usepackage{mathrsfs}
\usepackage{amsfonts}
\usepackage{graphicx}
\usepackage{latexsym, bm}

\usepackage{bbding}
\usepackage{tikz}
\usepackage{titleref}
\usepackage{subfigure}

\usepackage{enumerate}
\renewcommand{\labelenumi}{\rm(\theenumi).}
\renewcommand{\figurename}{Fig.}
\renewcommand\thesection{\arabic{section}}
\renewcommand{\baselinestretch}{1.2}

\newtheorem{theorem}{Theorem}[section]
\newtheorem{proposition}[theorem]{Proposition}
\newtheorem{lemma}[theorem]{Lemma}
\newtheorem{corollary}[theorem]{Corollary}

\theoremstyle{definition}

\theoremstyle{remark}
\newtheorem{remark}[theorem]{Remark}

\numberwithin{equation}{section}

\newcommand{\abs}[1]{\lvert#1\rvert}

\newcommand{\blankbox}[2]{%
  \parbox{\columnwidth}{\centering
    \setlength{\fboxsep}{0pt}%
    \fbox{\raisebox{0pt}[#2]{\hspace{#1}}}%
  }%
}

\usepackage{latexsym, bm}
\title
{{Number of Matchings of Low Order in (4,6)-Fullerene Graphs\thanks{This work was supported by NSFC (Grant No.11371180) and Cuiying Student Innovation Fund of Lanzhou University.} }}
\author{Zhi-Feng Wei$^{a,b}$\, Heping Zhang$^a$\thanks{
Corresponding author. \newline
\emph{E-mail address}: weizhifeng@vip.163.com,  zhanghp@lzu.edu.cn} \\
{\footnotesize $^{a}$School of Mathematics and Statistics, Lanzhou University, Lanzhou, Gansu 730000, P. R.~China}\\
{\footnotesize $^{b}$Cuiying Honors College, Lanzhou University, Lanzhou, Gansu 730000, P. R.~China}\\
}
\date{}

\begin{document}

\maketitle

\centerline{(Published by MATCH Commun.\ Math.\ Comput.\ Chem.\ in 2017)}

\begin{abstract}
    We obtain the formulae for the numbers of 4-matchings and 5-matchings  in terms of the number of hexagonal faces  in (4, 6)-fullerene graphs by studying structural classification  of 6-cycles and some local structural properties,  which correct the corresponding wrong results published.  Furthermore,  we obtain a  formula for the number of 6-matchings in tubular (4, 6)-fullerenes in terms of the number of  hexagonal  faces, and a formula for the number of 6-matchings in the other (4,6)-fullerenes  in terms of the numbers of  hexagonal  faces and  dual-squares.

    \baselineskip=0.30in

\end{abstract}

\section{Introduction}

    \thispagestyle{empty}

    Let $G$ be a graph with vertex set $V(G)$ and edge set $E(G)$. We use $n(G)$ and $m(G)$ to denote the numbers of vertices and edges of $G$ respectively. If $G$ is understood, then we use such notations without reference to $G$. A {\em $k$-matching} (or a {\em matching} of   order $k$) of $G$ is a set of $k$ pairwise nonadjacent edges. A {\em (4,6)-fullerene graph}, a mathematical model of a boron-nitrogen fullerene, is a  plane cubic graph with exactly six square faces and $\frac{n}{2}-4$ hexagonal faces. It is known that every (4,6)-fullerene graph is bipartite and 3-connected \cite{ZL10}.

    The number of matchings in a graph is of significance in theory and applications. The matching polynomial of a graph $G$ was defined as $\sum_{k=0}^{\nu}(-1)^kM(G,k)x^{n-2k}$ in  \cite{farrel79,gutman79}, where $\nu$ denotes the size of a maximum matching and $ M(G,k)$ the number of $k$-matchings. All roots of the matching polynomial are real, and the sequence $M(G,0), M(G,1),\ldots, M(G,\nu)$  is log-concave (cf. \cite{mtheory}). In 1971, Hosoya found a correlation between boiling point of some hydrocarbons  and  the total number of matchings in their molecular graphs~\cite{hosoya71}.
    \par
    There has already been some research into the enumeration of low order matchings. Klabjan and  Mohar \cite{kalabjan99} counted the matchings of  order at most 5 in hexagonal systems. Behmaram \cite{behmaram09} established a formula for the number of 4-matchings in triangle-free graphs with respect to the number of vertices, edges, degrees and 4-cycles. Vesalian and Asgari\cite{vesalian13}, and Vesalian et al. \cite{vesalian15} counted the 5-matchings and 6-matchings in graphs with girth at least 5 and 6, respectively. However, their methods and results are not applicable to  (4,6)-fullerene graphs.
    Behmaram et al. \cite{behmaram13} recently studied the number of matchings of order no more than 4 in (4,6)-fullerene graphs. But, all 6-cycles of (4,6)-fullerene graphs were mistakenly identified with hexagonal faces. We can give many (4,6)-fullerene graphs with 6-cycles which do not bound faces. This together with an error in counting the 5-length paths of (4,6)-fullerene graphs leads to wrong formulae for the numbers of 4-matchings \cite{behmaram13} and 5-matchings ~\cite{li14}.  So, we first give  a classification of 6-cycles of (4,6)-fullerene graphs in Theorem~\ref{struct6}. Then we correct the above errors and derive the enumeration result for 6-matchings of (4,6)-fullerene graphs. It turns out that there is not a unified  formula for the number of 6-matchings in (4,6)-fullerenes, so we give two formulae for 6-matchings in (4, 6)-fullerenes according to different graph structures.
    \par
    This paper is organized as follows. In section~\ref{sec_struct}, we study the structure of 6-cycles and obtain  a structural classification of (4,6)-fullerene graphs and some local properties. In section~\ref{sec_enum}, by establishing a series of recurrence relations, we deduce the enumeration  of higher order matchings to lower order ones. We obtain not only the correct formulae for the numbers of 4-matchings and 5-matchings but also the formulae for the numbers of 6-matchings. Meanwhile, we enumerate some other subgraphs of (4,6)-fullerene graphs.

\section{A classification of (4,6)-fullerene graphs}
    \label{sec_struct}

    Let $G$ be a graph. We say $G$ is  \emph{cyclically $k$-edge-connected} if $G$ cannot be separated into two components, each containing a cycle, by deletion of fewer than $k$ edges. We call the greatest integer $k$ such that $G$ is cyclically $k$-edge-connected the {\em cyclical edge-connectivity} of $G$, denoted by $\zeta(G)$. If $G$ is a (4,6)-fullerene graph, then $\zeta(G)=3$  or  4 (cf.~\cite{doslic03}). A tubular (4,6)-fullerene graph consists of $t\geqslant 0$ concentric layers of three hexagons, capped on each end by a cap formed by three squares. We denote such a tubular (4,6)-fullerene graph with $t$ layers of hexagonal faces by $T_t$. Figure~\ref{tubular} presents $T_3$ (Notice that the three dangling edges in the outside area actually connect to the same vertex). In the degenerate case $t = 0$, we get the ordinary cube. The family of all such tubes $T_t, t\geq 1$, is denoted by $\mathcal{T}$~\cite{doslic03}.
    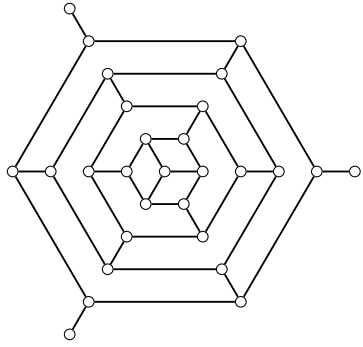
\begin{figure}[h]
        \centering
            \subfigure{
                \begin{tikzpicture}[scale=0.1]
                    \tikzstyle{vertex}=[draw,circle,minimum size=4pt,inner sep=0pt]
                    \tikzstyle{edge} = [draw,line width=0.7pt,-,black]
                    \tikzstyle{selected edge} = [draw,line width=2pt,-,black]
                    \node[vertex] (v0) at (0,0)  {};
                    \node[vertex] (v1) at (5,0)  {};
                    \node[vertex] (v2) at (2.5,4.32)   {};
                    \node[vertex] (v3) at (-2.5,4.32)   {};
                    \node[vertex] (v4) at (-5,0)  {};
                    \node[vertex] (v5) at (-2.5,-4.32)   {};
                    \node[vertex] (v6) at (2.5,-4.32)  {};
                    \node[vertex] (v7) at (10,0)  {};
                    \node[vertex] (v8) at (5,8.64)   {};
                    \node[vertex] (v9) at (-5,8.64)   {};
                    \node[vertex] (v10) at (-10,0)  {};
                    \node[vertex] (v11) at (-5,-8.64)   {};
                    \node[vertex] (v12) at (5,-8.64)  {};
                    \node[vertex] (v13) at (15,0)  {};
                    \node[vertex] (v14) at (7.5,12.96)   {};
                    \node[vertex] (v15) at (-7.5,12.96)   {};
                    \node[vertex] (v16) at (-15,0)  {};
                    \node[vertex] (v17) at (-7.5,-12.96)   {};
                    \node[vertex] (v18) at (7.5,-12.96)  {};
                    \node[vertex] (v19) at (20,0)  {};
                    \node[vertex] (v20) at (10,17.28)   {};
                    \node[vertex] (v21) at (-10,17.28)   {};
                    \node[vertex] (v22) at (-20,0)  {};
                    \node[vertex] (v23) at (-10,-17.28)   {};
                    \node[vertex] (v24) at (10,-17.28)  {};
                    \node[vertex] (v25) at (25,0)  {};
                    \node[vertex] (v27) at (-12.5,21.6)   {};
                    \node[vertex] (v29) at (-12.5,-21.6)   {};
                    \draw[edge] (v1)  -- (v2);
                    \draw[edge] (v2)  -- (v3);
                    \draw[edge] (v3)  -- (v4);
                    \draw[edge] (v4)  -- (v5);
                    \draw[edge] (v5)  -- (v6);
                    \draw[edge] (v6)  -- (v1);
                    \draw[edge] (v0)  -- (v1);
                    \draw[edge] (v0)  -- (v3);
                    \draw[edge] (v0)  -- (v5);
                    \draw[edge] (v7)  -- (v8);
                    \draw[edge] (v8)  -- (v9);
                    \draw[edge] (v9)  -- (v10);
                    \draw[edge] (v10)  -- (v11);
                    \draw[edge] (v11)  -- (v12);
                    \draw[edge] (v12)  -- (v7);
                    \draw[edge] (v2)  -- (v8);
                    \draw[edge] (v4)  -- (v10);
                    \draw[edge] (v6)  -- (v12);
                    \draw[edge] (v13)  -- (v14);
                    \draw[edge] (v14)  -- (v15);
                    \draw[edge] (v15)  -- (v16);
                    \draw[edge] (v16)  -- (v17);
                    \draw[edge] (v17)  -- (v18);
                    \draw[edge] (v18)  -- (v13);
                    \draw[edge] (v7)  -- (v13);
                    \draw[edge] (v9)  -- (v15);
                    \draw[edge] (v11)  -- (v17);
                    \draw[edge] (v19)  -- (v20);
                    \draw[edge] (v20)  -- (v21);
                    \draw[edge] (v21)  -- (v22);
                    \draw[edge] (v22)  -- (v23);
                    \draw[edge] (v23)  -- (v24);
                    \draw[edge] (v24)  -- (v19);
                    \draw[edge] (v14)  -- (v20);
                    \draw[edge] (v16)  -- (v22);
                    \draw[edge] (v18)  -- (v24);
                    \draw[edge] (v19)  -- (v25);
                    \draw[edge] (v21)  -- (v27);
                    \draw[edge] (v23)  -- (v29);
                \end{tikzpicture}}
                \caption{A tubular (4,6)-fullerene graph $T_3$ with~3 hexagon-layers.}
                \label{tubular}
        \end{figure}
    \par
    In a (4, 6)-fullerene graph $G$, $h(G)$ denotes the number of hexagonal faces.
    \begin{lemma}\label{eulercac}
    Let $G$ be a (4,6)-fullerene graph. Then we have\\
    (i)\cite{behmaram13} $n=2h+8, m=3h+12;$\\
    (ii) $h\not=1$;\\
    (iii) Any two faces of $G$ cannot have more than one common edge; and\\
    (iv) \cite{doslic03} $G\in\mathcal{T}$ if and only if $\zeta(G)=3$.
    \end{lemma}

    \begin{lemma}\cite{jiang11}\label{square}
            Let $G$ be a (4,6)-fullerene graph and $C$ a 4-cycle in $G$. Then $C$ must be a facial cycle of $G$.
    \end{lemma}

    \subsection{Structure of 6-cycles}

        For a (4,6)-fullerene graph $G$, Lemma~\ref{square} claims that a 4-cycle must be a facial cycle. But a 6-cycle is not necessarily a facial cycle. So we will study the structure of 6-cycles in (4,6)-fullerene graphs.
        \par
        We call the two subgraphs in Figure~\ref{typebf} dual-square and square-cap, respectively. For convenience, a boundary 6-cycle of a plane graph is  often represented by the name of this graph if this will not lead to confusion.
        \par
        By Lemma~\ref{eulercac}(iv) and the fact that $\zeta(G)=3 \text{ or } 4$, we only need to study the following two cases: $\zeta(G)=4$ (equivalently, $G\notin\mathcal{T}$) and $\zeta(G)=3$ (equivalently, $G\in\mathcal{T}$).

        \begin{lemma}\cite{jiang11}\label{6cyclejiang}
            Given a (4,6)-fullerene graph $G\notin\mathcal{T}$, if $C$ is a 6-cycle of $G$, then $C$ is either\\
            (i) a hexagonal facial cycle, or\\
            (ii) a dual-square, or\\
            (iii) a square-cap.
        \end{lemma}
        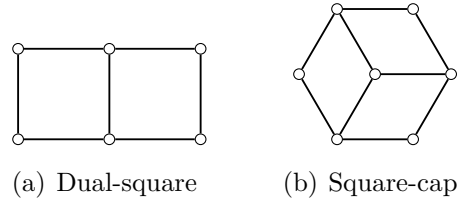
\begin{figure}
            \centering
            \subfigure[Dual-square]{
                \begin{tikzpicture}[scale=1.2]
                    \tikzstyle{vertex}=[draw,circle,minimum size=4pt,inner sep=0pt]
                    \tikzstyle{edge} = [draw,line width=0.7pt,-,black]
                    \node[vertex] (v1) at (0,0)  {};
                    \node[vertex] (v2) at (1,0)  {};
                    \node[vertex] (v3) at (2,0)  {};
                    \node[vertex] (v4) at (2,1)  {};
                    \node[vertex] (v5) at (1,1)  {};
                    \node[vertex] (v6) at (0,1)  {};
                    \draw[edge] (v1)  -- (v2);
                    \draw[edge] (v2)  -- (v3);
                    \draw[edge] (v3)  -- (v4);
                    \draw[edge] (v4)  -- (v5);
                    \draw[edge] (v5)  -- (v6);
                    \draw[edge] (v6)  -- (v1);
                    \draw[edge] (v2)  -- (v5);
                    \label{typeb1}
                \end{tikzpicture}}
                \hspace{4ex}
            \subfigure[Square-cap]{
                \begin{tikzpicture}[scale=0.2]
                    \tikzstyle{vertex}=[draw,circle,minimum size=4pt,inner sep=0pt]
                    \tikzstyle{edge} = [draw,line width=0.7pt,-,black]
                    \node[vertex] (v0) at (0,0)  {};
                    \node[vertex] (v1) at (5,0)  {};
                    \node[vertex] (v2) at (2.5,4.32)   {};
                    \node[vertex] (v3) at (-2.5,4.32)   {};
                    \node[vertex] (v4) at (-5,0)  {};
                    \node[vertex] (v5) at (-2.5,-4.32)   {};
                    \node[vertex] (v6) at (2.5,-4.32)  {};
                    \draw[edge] (v1)  -- (v2);
                    \draw[edge] (v2)  -- (v3);
                    \draw[edge] (v3)  -- (v4);
                    \draw[edge] (v4)  -- (v5);
                    \draw[edge] (v5)  -- (v6);
                    \draw[edge] (v6)  -- (v1);
                    \draw[edge] (v0)  -- (v1);
                    \draw[edge] (v0)  -- (v3);
                    \draw[edge] (v0)  -- (v5);
                    \label{typeb2}
                \end{tikzpicture}}
            \caption{Two types of 6-cycle in (4,6)-fullerene graphs.}
            \label{typebf}
        \end{figure}
        \begin{lemma}\label{typegc}
            For a (4,6)-fullerene graph $G$ with $h\geqslant 2$, if $G$ has a square-cap subgraph, then $\zeta(G)=3$ and  $G$ is thus tubular.
        \end{lemma}
        \begin{proof}
            This is a direct inference from Lemma~\ref{eulercac}(iv).
        \end{proof}
        By Lemma~\ref{typegc}, we can strengthen Lemma~\ref{6cyclejiang} as follows.
        \begin{lemma}\label{6cycle1}
            For a (4,6)-fullerene graph $G\notin\mathcal{T}$ other than the cube, if $C$ is a 6-cycle of $G$, then $C$ is either a hexagonal facial cycle or a dual-square.
        \end{lemma}
        We call the graph in Figure~\ref{tubularfigure} ``a square-cap with 2 hexagon-layers". Generally, we have ``a square cap with $k$ hexagon-layers", where $k\geqslant1$.
        \begin{lemma}\label{6cycle2}
            Given a tubular (4,6)-fullerene graph $T_t$ with $t\geqslant 1$, if $C$ is a 6-cycle of $T_t$, then $C$ is either\\
            (i) a hexagonal facial cycle, or\\
            (ii) a dual-square, or\\
            (iii) a square-cap, or\\
            (iv) the boundary of a square-cap with hexagon-layers.
        \end{lemma}
        \begin{figure}[h]
            \centering
                \begin{tikzpicture}[scale=0.1]
                    \tikzstyle{vertex}=[draw,circle,minimum size=4pt,inner sep=0pt]
                    \tikzstyle{edge} = [draw,line width=0.7pt,-,black]
                    \node[vertex] (v0) at (0,0)  {};
                    \node[vertex] (v1) at (5,0)  {};
                    \node[vertex] (v2) at (2.5,4.32)   {};
                    \node[vertex] (v3) at (-2.5,4.32)   {};
                    \node[vertex] (v4) at (-5,0)  {};
                    \node[vertex] (v5) at (-2.5,-4.32)   {};
                    \node[vertex] (v6) at (2.5,-4.32)  {};
                    \node[vertex] (v7) at (10,0)  {};
                    \node[vertex] (v8) at (5,8.64)   {};
                    \node[vertex] (v9) at (-5,8.64)   {};
                    \node[vertex] (v10) at (-10,0)  {};
                    \node[vertex] (v11) at (-5,-8.64)   {};
                    \node[vertex] (v12) at (5,-8.64)  {};
                    \node[vertex] (v13) at (15,0)  {};
                    \node[vertex] (v14) at (7.5,12.96)   {};
                    \node[vertex] (v15) at (-7.5,12.96)   {};
                    \node[vertex] (v16) at (-15,0)  {};
                    \node[vertex] (v17) at (-7.5,-12.96)   {};
                    \node[vertex] (v18) at (7.5,-12.96)  {};
                    \draw[edge] (v1)  -- (v2);
                    \draw[edge] (v2)  -- (v3);
                    \draw[edge] (v3)  -- (v4);
                    \draw[edge] (v4)  -- (v5);
                    \draw[edge] (v5)  -- (v6);
                    \draw[edge] (v6)  -- (v1);
                    \draw[edge] (v0)  -- (v1);
                    \draw[edge] (v0)  -- (v3);
                    \draw[edge] (v0)  -- (v5);
                    \draw[edge] (v7)  -- (v8);
                    \draw[edge] (v8)  -- (v9);
                    \draw[edge] (v9)  -- (v10);
                    \draw[edge] (v10)  -- (v11);
                    \draw[edge] (v11)  -- (v12);
                    \draw[edge] (v12)  -- (v7);
                    \draw[edge] (v2)  -- (v8);
                    \draw[edge] (v4)  -- (v10);
                    \draw[edge] (v6)  -- (v12);
                    \draw[edge] (v13)  -- (v14);
                    \draw[edge] (v14)  -- (v15);
                    \draw[edge] (v15)  -- (v16);
                    \draw[edge] (v16)  -- (v17);
                    \draw[edge] (v17)  -- (v18);
                    \draw[edge] (v18)  -- (v13);
                    \draw[edge] (v7)  -- (v13);
                    \draw[edge] (v9)  -- (v15);
                    \draw[edge] (v11)  -- (v17);
                \end{tikzpicture}
                \caption{A square-cap with 2 hexagon-layers.}
                \label{tubularfigure}
        \end{figure}
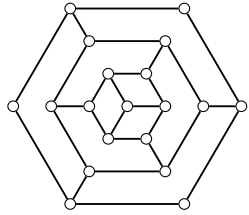
        \begin{proof}
            We first give partitions of $E(T_t)$ and $V(T_t)$ according to the concentric layers of $T_t$. We first fix one of the two square-caps, and the 0-layer  is the set of edges of the fixed square-cap. The 1-layer is the boundary edges of the square-cap with~1 hexagon-layer. In general, we define the $i$-layer as the boundary edges of the  square-cap with $i$ hexagon-layers for $1\leqslant i\leqslant t-1$. The $t$-layer of  $T_t$ is defined as the set of edges of the other square-cap in $T_t$. Thus, the $i$-layer is an edge set $L_i$, $0\leqslant i\leqslant t$. We define $\widetilde{L_i}$ as the set of vertices  incident with edges in $L_i$. Then $V(T_t)=\cup_{i=0}^t \widetilde{L}_i$. We call the set of edges linking $\widetilde{L}_i$ and $\widetilde{L}_{i+1}$ the $i$-bridge $B_i$, $0\leqslant i\leqslant t-1$. Thus $E(T_t)=(\cup_{i=0}^t L_i) \cup (\cup_{j=0}^{t-1} B_j)$.
            \par
            If $E(C)\subseteq L_i$ for some $0\leqslant i\leqslant t$, one can easily check this result. Otherwise, $E(C)\cap B_i\neq\emptyset $ for some $i$, $0\leqslant i\leqslant t-1$. We claim that $C$ is a hexagonal facial cycle.
            \par
            Let $C=u_1u_2u_3u_4u_5u_6u_1$. Without loss of generality, we suppose that $u_1u_2\in  E(C)\cap B_i$. Since $B_i$ is an edge-cut, $C$ has  another edge in $B_i$, say $vw$, where $v,w\in V(C)\backslash\{u_1,u_2\}$. Further we assume that $u_1, v\in \widetilde{L}_i$ and $u_2, w\in \widetilde{L}_{i+1}$. We now consider the distance $d(u, v)$ between  vertices $u$ and $v$ of $T_t$ (the length of a shortest path between $u$ and $v$). It is obvious that $d(u_1,v)=d(u_2,w)=2$. Since the common neighbour of $u_1$ and $v$ (respectively, $u_2$ and $w$) is unique, the 2-path linking $u_1$ and $v$ (respectively, $u_2$ and $w$) is also unique and thus located in $L_i$ ($L_{i+1}$, respectively). This implies $|E(C)\cap L_i|=|E(C)\cap L_{i+1}|=2$. Hence we have $w=u_4$, $v=u_5$ and $C$ is a hexagonal facial cycle.
        \end{proof}

        Combining Lemmas~\ref{6cycle1} and~\ref{6cycle2}, we have
        \begin{theorem}[Structure of 6-cycles]\label{struct6}
            Let $G$ be a (4,6)-fullerene graph and $C$ a 6-cycle in $G$. Then $C$ is a hexagonal facial cycle, a dual-square, a square-cap or a square-cap with hexagon-layers.
        \end{theorem}
    \subsection{Local structural properties}
        Lemma~\ref{typegc} is interesting since we can determine the global structure of the graph from its local behavior. Analogous situations can be found in the following two propositions.

        \begin{proposition}\label{lotogll}
            Let $G$ be a (4,6)-fullerene graph that has at least 4 square faces arranged in a line. Then $G$ is either a cube or a hexagonal prism.
        \end{proposition}
        \begin{figure}[ht]
            \centering
            \subfigure[]{
                \begin{tikzpicture}[scale=1]
                    \tikzstyle{vertex}=[draw,circle,minimum size=4pt,inner sep=0pt]
                    \tikzstyle{edge} = [draw,line width=0.7pt,-,black]
                    \node[vertex] (v1) at (0,1)[label=above:$v_1$]{};
                    \node[vertex] (v2) at (1,1)[label=above:$v_2$]{};
                    \node[vertex] (v3) at (2,1)[label=above:$v_3$]{};
                    \node[vertex] (v4) at (3,1)[label=above:$v_4$]{};
                    \node[vertex] (v5) at (4,1)[label=above:$v_5$]{};
                    \node[vertex] (v6) at (0,0)[label=below:$u_1$]{};
                    \node[vertex] (v7) at (1,0)[label=below:$u_2$] {};
                    \node[vertex] (v8) at (2,0)[label=below:$u_3$] {};
                    \node[vertex] (v9) at (3,0)[label=below:$u_4$]  {};
                    \node[vertex] (v10) at (4,0) [label=below:$u_5$]{};
                    \draw[edge] (v1)  -- (v2);
                    \draw[edge] (v2)  -- (v3);
                    \draw[edge] (v3)  -- (v4);
                    \draw[edge] (v4)  -- (v5);
                    \draw[edge] (v6)  -- (v7);
                    \draw[edge] (v7)  -- (v8);
                    \draw[edge] (v8)  -- (v9);
                    \draw[edge] (v9)  -- (v10);
                    \draw[edge] (v1)  -- (v6);
                    \draw[edge] (v2)  -- (v7);
                    \draw[edge] (v3)  -- (v8);
                    \draw[edge] (v4)  -- (v9);
                    \draw[edge] (v5)  -- (v10);
                    \label{ltgfc1}
                \end{tikzpicture}}
             \hspace{3ex}
             \subfigure[]{
                \begin{tikzpicture}[scale=0.8]
                    \tikzstyle{vertex}=[draw,circle,minimum size=4pt,inner sep=0pt]
                    \tikzstyle{edge} = [draw,line width=0.7pt,-,black]
                    \node[vertex] (v1) at (0,1)[label=above:$v_1$]{};
                    \node[vertex] (v2) at (1,1){};
                    \node[vertex] (v3) at (2,1){};
                    \node[vertex] (v4) at (3,1){};
                    \node[vertex] (v5) at (4,1)[label=above:$v_5$]{};
                    \node[vertex] (v6) at (0,0)[label=below:$u_1$]{};
                    \node[vertex] (v7) at (1,0){};
                    \node[vertex] (v8) at (2,0){};
                    \node[vertex] (v9) at (3,0){};
                    \node[vertex] (v10) at (4,0) [label=below:$u_5$]{};
                    \node[vertex] (x1) at (2,1.8)  [label=above:$x_1$]{};
                    \node[vertex] (x2) at (2,-1) [label=above:$x_2$]{};
                    \draw[edge] (v1)  -- (v2);
                    \draw[edge] (v2)  -- (v3);
                    \draw[edge] (v3)  -- (v4);
                    \draw[edge] (v4)  -- (v5);
                    \draw[edge] (v6)  -- (v7);
                    \draw[edge] (v7)  -- (v8);
                    \draw[edge] (v8)  -- (v9);
                    \draw[edge] (v9)  -- (v10);
                    \draw[edge] (v1)  -- (v6);
                    \draw[edge] (v2)  -- (v7);
                    \draw[edge] (v3)  -- (v8);
                    \draw[edge] (v4)  -- (v9);
                    \draw[edge] (v5)  -- (v10);
                    \draw[edge] (v1)  -- (x1);
                    \draw[edge] (v5)  -- (x1);
                    \draw[edge] (v6)  -- (x2);
                    \draw[edge] (v10)  -- (x2);
                    \label{ltgfc2}
                \end{tikzpicture}}
             \hspace{3ex}
             \subfigure[]{
                \begin{tikzpicture}[scale=0.12]
                            \tikzstyle{vertex}=[draw,circle,minimum size=4pt,inner sep=0pt]
                            \tikzstyle{edge} = [draw,line width=0.7pt,-,black]
                            \node[vertex] (v1) at (5,0)  {};
                            \node[vertex] (v2) at (2.5,4.32)   {};
                            \node[vertex] (v3) at (-2.5,4.32)   {};
                            \node[vertex] (v4) at (-5,0)  {};
                            \node[vertex] (v5) at (-2.5,-4.32)   {};
                            \node[vertex] (v6) at (2.5,-4.32)  {};
                            \node[vertex] (v7) at (10,0)  {};
                            \node[vertex] (v8) at (5,8.64)   {};
                            \node[vertex] (v9) at (-5,8.64)   {};
                            \node[vertex] (v10) at (-10,0)  {};
                            \node[vertex] (v11) at (-5,-8.64)   {};
                            \node[vertex] (v12) at (5,-8.64)  {};
                            \draw[edge] (v1)  -- (v2);
                            \draw[edge] (v2)  -- (v3);
                            \draw[edge] (v3)  -- (v4);
                            \draw[edge] (v4)  -- (v5);
                            \draw[edge] (v5)  -- (v6);
                            \draw[edge] (v6)  -- (v1);
                            \draw[edge] (v7)  -- (v8);
                            \draw[edge] (v8)  -- (v9);
                            \draw[edge] (v9)  -- (v10);
                            \draw[edge] (v10)  -- (v11);
                            \draw[edge] (v11)  -- (v12);
                            \draw[edge] (v12)  -- (v7);
                            \draw[edge] (v1)  -- (v7);
                            \draw[edge] (v2)  -- (v8);
                            \draw[edge] (v3)  -- (v9);
                            \draw[edge] (v4)  -- (v10);
                            \draw[edge] (v5)  -- (v11);
                            \draw[edge] (v6)  -- (v12);
                            \label{ltgfc3}
                        \end{tikzpicture}}
            \caption{Graphs for proof of Proposition~\ref{lotogll}.}\label{lotogllc}
        \end{figure}
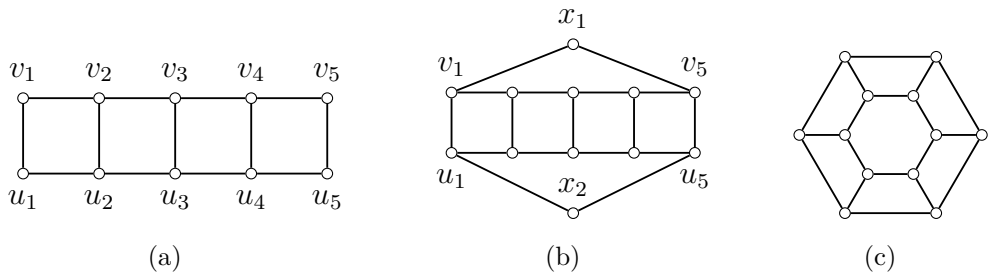
        \begin{proof}
            \par
            Suppose that $G$ is not a cube and has a subgraph as depicted in Figure~\ref{lotogllc}(a). Next we show that $G$ is a hexagonal prism.
            \par
            Since $G$ is a bipartite graph, $v_1v_5\notin E(G)$, and the path $v_1v_2v_3v_4v_5$ lies on the boundary of a hexagonal face of $G$. That means that $G$ has a vertex $x_1$ adjacent to both $v_1$ and $v_5$ such that the 6-cycle $x_1v_1v_2v_3v_4v_5x_1$ is a hexagonal face. Similarly, there is another vertex $x_2$ adjacent to both $u_1$ and $u_5$  such that $x_2u_1u_2u_3u_4u_5x_2$ is a hexagonal face (see Figure~\ref{ltgfc2}). Now, we get a subgraph of $G$ as shown in Figure~\ref{ltgfc2} whose boundary is 6-cycle $x_1v_1u_1x_2u_5v_5x_1$. There are exactly two vertices of degree 2 on this 6-cycle. By Theorem  \ref{struct6} this 6-cycle must be a dual-square and $x_1x_2\in E(G)$. Now $G$ is a hexagonal prism (Figure~\ref{ltgfc3}).
        \end{proof}

        \begin{proposition}\label{lotogl}
             For a (4,6)-fullerene graph $G$, if there are some  three square faces arranged in a line, then the other three square faces are also arranged in a line.
        \end{proposition}
        \begin{figure}[ht]
            \centering
                \subfigure[]{
                    \begin{tikzpicture}[scale=0.07]
                    \tikzstyle{vertex}=[draw,circle,minimum size=4pt,inner sep=0pt]
                    \tikzstyle{edge} = [draw,line width=0.7pt,-,black]
                    \node[vertex] (v0) at (-15,5)[label=above:$v_1$]{};
                    \node[vertex] (v1) at (-5,5)[label=above:$v_2$]{};
                    \node[vertex] (v2) at (5,5)[label=above:$v_3$]{}{};
                    \node[vertex] (v3) at (15,5)[label=above:$v_4$]{};
                    \node[vertex] (v4) at (15,-5)[label=below:$v_5$]{};
                    \node[vertex] (v5) at (5,-5)[label=below:$v_6$]{};
                    \node[vertex] (v6) at (-5,-5)[label=below:$v_7$]{};
                    \node[vertex] (v7) at (-15,-5)[label=below:$v_8$]{};
                    \draw[edge]  (v0) to (v1);
                    \draw[edge]  (v1) to (v2);
                    \draw[edge]  (v2) to (v3);
                    \draw[edge]  (v3) to (v4);
                    \draw[edge]  (v4) to (v5);
                    \draw[edge]  (v5) to (v6);
                    \draw[edge]  (v6) to (v7);
                    \draw[edge]  (v7) to (v0);
                    \draw[edge]  (v1) to (v6);
                    \draw[edge]  (v2) to (v5);
                    \label{ltgp1}
                    \end{tikzpicture}}
                \subfigure[]{
                    \begin{tikzpicture}[scale=0.07]
                    \tikzstyle{vertex}=[draw,circle,minimum size=4pt,inner sep=0pt]
                    \tikzstyle{edge} = [draw,line width=0.7pt,-,black]
                    \node[vertex] (v0) at (-15,5)[label=above:$v_1$]{};
                    \node[vertex] (v1) at (-5,5)[label=above:$v_2$]{};
                    \node[vertex] (v2) at (5,5)[label=above:$v_3$]{}{};
                    \node[vertex] (v3) at (15,5)[label=above:$v_4$]{};
                    \node[vertex] (v4) at (15,-5)[label=below:$v_5$]{};
                    \node[vertex] (v5) at (5,-5)[label=below:$v_6$]{};
                    \node[vertex] (v6) at (-5,-5)[label=below:$v_7$]{};
                    \node[vertex] (v7) at (-15,-5)[label=below:$v_8$]{};
                    \draw[edge]  (v0) to (v1);
                    \draw[edge]  (v1) to (v2);
                    \draw[edge]  (v2) to (v3);
                    \draw[edge]  (v3) to (v4);
                    \draw[edge]  (v4) to (v5);
                    \draw[edge]  (v5) to (v6);
                    \draw[edge]  (v6) to (v7);
                    \draw[edge]  (v7) to (v0);
                    \draw[edge]  (v1) to (v6);
                    \draw[edge]  (v2) to (v5);
                    \draw[line width=0.7pt] (v0) sin(0,15) cos (v3);
                    \label{ltgp2}
                    \end{tikzpicture}}
                \subfigure[]{
                    \begin{tikzpicture}[scale=0.07]
                    \tikzstyle{vertex}=[draw,circle,minimum size=4pt,inner sep=0pt]
                    \tikzstyle{edge} = [draw,line width=0.7pt,-,black]
                    \node[vertex] (v0) at (-15,5)[label=above:$v_1$]{};
                    \node[vertex] (v1) at (-5,5)[label=above:$v_2$]{};
                    \node[vertex] (v2) at (5,5)[label=above:$v_3$]{}{};
                    \node[vertex] (v3) at (15,5)[label=above:$v_4$]{};
                    \node[vertex] (v4) at (15,-5)[label=below:$v_5$]{};
                    \node[vertex] (v5) at (5,-5)[label=below:$v_6$]{};
                    \node[vertex] (v6) at (-5,-5)[label=below:$v_7$]{};
                    \node[vertex] (v7) at (-15,-5)[label=below:$v_8$]{};
                    \draw[edge]  (v0) to (v1);
                    \draw[edge]  (v1) to (v2);
                    \draw[edge]  (v2) to (v3);
                    \draw[edge]  (v3) to (v4);
                    \draw[edge]  (v4) to (v5);
                    \draw[edge]  (v5) to (v6);
                    \draw[edge]  (v6) to (v7);
                    \draw[edge]  (v7) to (v0);
                    \draw[edge]  (v1) to (v6);
                    \draw[edge]  (v2) to (v5);
                    \node[vertex] (v12) at (-23,13)[label=left:$u_1$]{};
                    \node[vertex] (v15) at (-23,-13)[label=left:$u_2$]{};
                    \draw[edge]  (v0) to (v12);
                    \draw[edge]  (v15) to (v7);
                    \label{ltgp4}
                    \end{tikzpicture}}
                \subfigure[]{
                    \begin{tikzpicture}[scale=0.07]
                    \tikzstyle{vertex}=[draw,circle,minimum size=4pt,inner sep=0pt]
                    \tikzstyle{edge} = [draw,line width=0.7pt,-,black]
                    \node[vertex] (v0) at (-15,5)[label=left:$v_1$]{};
                    \node[vertex] (v1) at (-5,5){};
                    \node[vertex] (v2) at (5,5){};
                    \node[vertex] (v3) at (15,5)[label=right:$v_4$]{};
                    \node[vertex] (v4) at (15,-5)[label=right:$v_5$]{};
                    \node[vertex] (v5) at (5,-5){};
                    \node[vertex] (v6) at (-5,-5){};
                    \node[vertex] (v7) at (-15,-5)[label=left:$v_8$]{};
                    \draw[edge]  (v0) to (v1);
                    \draw[edge]  (v1) to (v2);
                    \draw[edge]  (v2) to (v3);
                    \draw[edge]  (v3) to (v4);
                    \draw[edge]  (v4) to (v5);
                    \draw[edge]  (v5) to (v6);
                    \draw[edge]  (v6) to (v7);
                    \draw[edge]  (v7) to (v0);
                    \draw[edge]  (v1) to (v6);
                    \draw[edge]  (v2) to (v5);
                    \node[vertex] (v8) at (23,13)[label=right:$u_3$]{};
                    \node[vertex] (v11) at (23,-13)[label=right:$u_4$]{};
                    \draw[edge]  (v3) to (v8);
                    \draw[edge]  (v11) to (v4);
                    \node[vertex] (v12) at (-23,13)[label=left:$u_1$]{};
                    \node[vertex] (v15) at (-23,-13)[label=left:$u_2$]{};
                    \draw[edge]  (v0) to (v12);
                    \draw[edge]  (v15) to (v7);
                    \label{ltgp5}
                    \end{tikzpicture}}\\
                \subfigure[]{
                    \begin{tikzpicture}[scale=0.05]
                    \tikzstyle{vertex}=[draw,circle,minimum size=4pt,inner sep=0pt]
                    \tikzstyle{edge} = [draw,line width=0.7pt,-,black]
                    \node[vertex] (v0) at (-15,5)[label=left:$v_1$]{};
                    \node[vertex] (v1) at (-5,5){};
                    \node[vertex] (v2) at (5,5){};
                    \node[vertex] (v3) at (15,5)[label=right:$v_4$]{};
                    \node[vertex] (v4) at (15,-5)[label=right:$v_5$]{};
                    \node[vertex] (v5) at (5,-5){};
                    \node[vertex] (v6) at (-5,-5){};
                    \node[vertex] (v7) at (-15,-5)[label=left:$v_8$]{};
                    \draw[edge]  (v0) to (v1);
                    \draw[edge]  (v1) to (v2);
                    \draw[edge]  (v2) to (v3);
                    \draw[edge]  (v3) to (v4);
                    \draw[edge]  (v4) to (v5);
                    \draw[edge]  (v5) to (v6);
                    \draw[edge]  (v6) to (v7);
                    \draw[edge]  (v7) to (v0);
                    \draw[edge]  (v1) to (v6);
                    \draw[edge]  (v2) to (v5);
                    \node[vertex] (v8) at (23,13)[label=right:$u_3$]{};
                    \node[vertex] (v11) at (23,-13)[label=right:$u_4$]{};
                    \draw[edge]  (v3) to (v8);
                    \draw[edge]  (v11) to (v4);
                    \node[vertex] (v12) at (-23,13)[label=left:$u_1$]{};
                    \node[vertex] (v15) at (-23,-13)[label=left:$u_2$]{};
                    \draw[edge]  (v0) to (v12);
                    \draw[edge]  (v15) to (v7);
                    \draw[edge]  (v12) to (v8);
                    \draw[edge]  (v15) to (v11);
                    \label{ltgp6}
                    \end{tikzpicture}}
                \subfigure[]{
                    \begin{tikzpicture}[scale=0.05]
                    \tikzstyle{vertex}=[draw,circle,minimum size=4pt,inner sep=0pt]
                    \tikzstyle{edge} = [draw,line width=0.7pt,-,black]
                    \node[vertex] (v0) at (-15,5)[label=left:$v_1$]{};
                    \node[vertex] (v1) at (-5,5){};
                    \node[vertex] (v2) at (5,5){};
                    \node[vertex] (v3) at (15,5)[label=right:$v_4$]{};
                    \node[vertex] (v4) at (15,-5)[label=right:$v_5$]{};
                    \node[vertex] (v5) at (5,-5){};
                    \node[vertex] (v6) at (-5,-5){};
                    \node[vertex] (v7) at (-15,-5)[label=left:$v_8$]{};
                    \draw[edge]  (v0) to (v1);
                    \draw[edge]  (v1) to (v2);
                    \draw[edge]  (v2) to (v3);
                    \draw[edge]  (v3) to (v4);
                    \draw[edge]  (v4) to (v5);
                    \draw[edge]  (v5) to (v6);
                    \draw[edge]  (v6) to (v7);
                    \draw[edge]  (v7) to (v0);
                    \draw[edge]  (v1) to (v6);
                    \draw[edge]  (v2) to (v5);
                    \node[vertex] (v8) at (23,13)[label=above:$u_3$]{};
                    \node[vertex] (v9) at (31,18)[label=right:$w_3$]{};
                    \node[vertex] (v10) at (31,-18)[label=right:$w_4$]{};
                    \node[vertex] (v11) at (23,-13)[label=below:$u_4$]{};
                    \draw[edge]  (v3) to (v8);
                    \draw[edge]  (v8) to (v9);
                    \draw[edge]  (v10) to (v11);
                    \draw[edge]  (v11) to (v4);
                    \node[vertex] (v12) at (-23,13)[label=above:$u_1$]{};
                    \node[vertex] (v13) at (-31,18)[label=left:$w_1$]{};
                    \node[vertex] (v14) at (-31,-18)[label=left:$w_2$]{};
                    \node[vertex] (v15) at (-23,-13)[label=below:$u_2$]{};
                    \draw[edge]  (v0) to (v12);
                    \draw[edge]  (v12) to (v13);
                    \draw[edge]  (v14) to (v15);
                    \draw[edge]  (v15) to (v7);
                    \draw[edge]  (v12) to (v8);
                    \draw[edge]  (v15) to (v11);
                    \label{ltgp7}
                    \end{tikzpicture}}
                \subfigure[]{
                    \begin{tikzpicture}[scale=0.05]
                    \tikzstyle{vertex}=[draw,circle,minimum size=4pt,inner sep=0pt]
                    \tikzstyle{edge} = [draw,line width=0.7pt,-,black]
                    \node[vertex] (v0) at (-15,5)[label=left:$v_1$]{};
                    \node[vertex] (v1) at (-5,5){};
                    \node[vertex] (v2) at (5,5){};
                    \node[vertex] (v3) at (15,5)[label=right:$v_4$]{};
                    \node[vertex] (v4) at (15,-5)[label=right:$v_5$]{};
                    \node[vertex] (v5) at (5,-5){};
                    \node[vertex] (v6) at (-5,-5){};
                    \node[vertex] (v7) at (-15,-5)[label=left:$v_8$]{};
                    \draw[edge]  (v0) to (v1);
                    \draw[edge]  (v1) to (v2);
                    \draw[edge]  (v2) to (v3);
                    \draw[edge]  (v3) to (v4);
                    \draw[edge]  (v4) to (v5);
                    \draw[edge]  (v5) to (v6);
                    \draw[edge]  (v6) to (v7);
                    \draw[edge]  (v7) to (v0);
                    \draw[edge]  (v1) to (v6);
                    \draw[edge]  (v2) to (v5);
                    \node[vertex] (v8) at (23,13)[label=above:$u_3$]{};
                    \node[vertex] (v9) at (31,18)[label=right:$w_3$]{};
                    \node[vertex] (v10) at (31,-18)[label=right:$w_4$]{};
                    \node[vertex] (v11) at (23,-13)[label=below:$u_4$]{};
                    \draw[edge]  (v3) to (v8);
                    \draw[edge]  (v8) to (v9);
                    \draw[edge]  (v9) to (v10);
                    \draw[edge]  (v10) to (v11);
                    \draw[edge]  (v11) to (v4);
                    \node[vertex] (v12) at (-23,13)[label=above:$u_1$]{};
                    \node[vertex] (v13) at (-31,18)[label=left:$w_1$]{};
                    \node[vertex] (v14) at (-31,-18)[label=left:$w_2$]{};
                    \node[vertex] (v15) at (-23,-13)[label=below:$u_2$]{};
                    \draw[edge]  (v0) to (v12);
                    \draw[edge]  (v12) to (v13);
                    \draw[edge]  (v13) to (v14);
                    \draw[edge]  (v14) to (v15);
                    \draw[edge]  (v15) to (v7);
                    \draw[edge]  (v12) to (v8);
                    \draw[edge]  (v15) to (v11);
                    \label{ltgp8}
                    \end{tikzpicture}}
                    \caption{Graphs for proof of Proposition~\ref{lotogl}.}
            \end{figure}
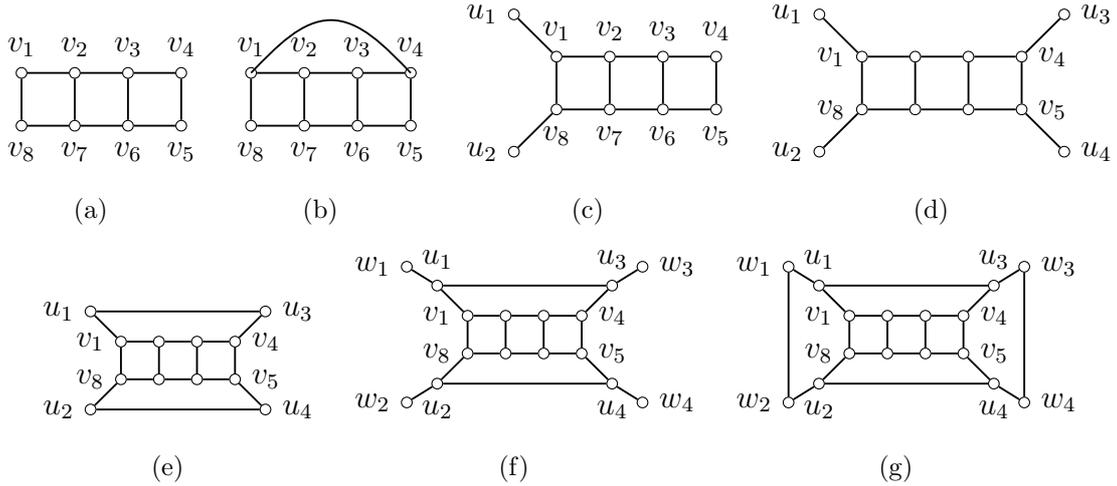

        \begin{proof} If there are more than 3 square faces arranged in a line, Proposition 2.8 implies that  $G$ is either a cube or a hexagonal prism and  the conclusion holds. So suppose that $G$ has some exactly three square faces arranged in a line. That is, the referred 3 square faces form a subgraph as shown in Figure~\ref{ltgp1} whose vertices are labelled so that each of edges $v_1v_8$ and $v_4v_5$ belongs to a hexagonal face.  Let $C=v_1v_2v_3v_4v_5v_6v_7v_8v_1$.

        Because $G$ is bipartite, $v_1v_5, v_4v_8\notin E(G)$. If $v_1v_4\in E(G)$ (Figure~\ref{ltgp2}), then  Theorem \ref{struct6} implies that 6-cycle $v_1v_4v_5v_6v_7v_8v_1$ is a dual-square, so $v_5v_8\in E(G)$. That is, $v_1v_8$ belongs to two squares, a contradiction. Hence $v_1v_4\notin E(G)$, similarly,  $v_5v_8\notin E(G)$.
            \par
            Therefore, there are vertices  $u_1$ and $u_2$  which are adjacent to $v_1$ and  $v_8$  respectively (Figure~\ref{ltgp4}). If  $u_2v_4\in E(G)$, 6-cycle $v_4v_5v_6v_7v_8u_2v_4$ is a dual-square by Lemma~\ref{6cycle1}, but this is impossible. Hence $u_2v_4\notin E(G)$, and similarly $u_1v_5\notin E(G)$. So $G$ has a vertex $u_3$ adjacent to $v_4$,  and  a vertex $u_4$ adjacent to $v_5$ (Figure~\ref{ltgp5}). Since $G$ is a (4,6)fullerene, $u_1u_3, u_2u_4\in E(G)$, and  $G$ has a subgraph as in Figure~\ref{ltgp6}. Since  $v_1v_8$ and $v_4v_5$  belong  to a hexagonal face,  $u_1u_2, u_3u_4\notin E(G)$. As above, $G$ has vertices $w_1, w_2, w_3, w_4$ satisfying $w_iu_i\in E(G), i=1,2,3,4$ (Figure~\ref{ltgp7}). Further, we have $w_1w_2, w_3w_4\in E(G)$ (Figure~\ref{ltgp8}). Since 8-cycle $w_1u_1u_3w_3w_4u_4u_2w_2w_1$ has a similar structure with $C$. If $w_1w_3\in E(G)$, then $w_2w_4\in E(G)$. In this case graph $G$ has already been determined (Figure~\ref{lanternsample}(a)), and the theorem holds. If $w_1w_3\notin E(G)$, then $w_2w_4\notin E(G)$.  The theorem also follows from  analogous arguments as above.
        \end{proof}

        \begin{remark}
            In the above proof, we have actually ascertained the whole graph structure. To be precise, we have exactly two groups of squares. In each group there are 3 squares arranged in a line.  If such two groups of squares have some common edges, the corresponding (4,6)-fullerene graphs must be cube and hexagonal prism. Otherwise, hexagonal faces are distributed around the two square groups of squares (See Figure~\ref{lanternsample} for two examples).  For ease of terminology, we say the latter kind of (4,6)-fullerene graphs are of ``lantern structure", which is important in the following enumeration.
        \end{remark}

        \begin{figure}\centering
                    \subfigure[]{
                        \begin{tikzpicture}[scale=0.055]
                        \tikzstyle{vertex}=[draw,circle,minimum size=4pt,inner sep=0pt]
                        \tikzstyle{edge} = [draw,line width=0.7pt,-,black]
                        \node[vertex] (v0) at (-15,5){};
                        \node[vertex] (v1) at (-5,5){};
                        \node[vertex] (v2) at (5,5){};
                        \node[vertex] (v3) at (15,5){};
                        \node[vertex] (v4) at (15,-5){};
                        \node[vertex] (v5) at (5,-5){};
                        \node[vertex] (v6) at (-5,-5){};
                        \node[vertex] (v7) at (-15,-5){};
                        \draw[edge]  (v0) to (v1);
                        \draw[edge]  (v1) to (v2);
                        \draw[edge]  (v2) to (v3);
                        \draw[edge]  (v3) to (v4);
                        \draw[edge]  (v4) to (v5);
                        \draw[edge]  (v5) to (v6);
                        \draw[edge]  (v6) to (v7);
                        \draw[edge]  (v7) to (v0);
                        \draw[edge]  (v1) to (v6);
                        \draw[edge]  (v2) to (v5);
                        \node[vertex] (v8) at (23,13){};
                        \node[vertex] (v9) at (31,18){};
                        \node[vertex] (v10) at (31,-18){};
                        \node[vertex] (v11) at (23,-13){};
                        \draw[edge]  (v3) to (v8);
                        \draw[edge]  (v8) to (v9);
                        \draw[edge]  (v9) to (v10);
                        \draw[edge]  (v10) to (v11);
                        \draw[edge]  (v11) to (v4);
                        \node[vertex] (v12) at (-23,13){};
                        \node[vertex] (v13) at (-31,18){};
                        \node[vertex] (v14) at (-31,-18){};
                        \node[vertex] (v15) at (-23,-13){};
                        \draw[edge]  (v0) to (v12);
                        \draw[edge]  (v12) to (v13);
                        \draw[edge]  (v13) to (v14);
                        \draw[edge]  (v14) to (v15);
                        \draw[edge]  (v15) to (v7);
                        \draw[edge]  (v12) to (v8);
                        \draw[edge]  (v15) to (v11);
                        \draw[edge]  (v9) to (v13);
                        \draw[edge]  (v10) to (v14);
                    \end{tikzpicture}}
                \hspace{5ex}
                    \subfigure[]{
                        \begin{tikzpicture}[scale=0.04]
                        \tikzstyle{vertex}=[draw,circle,minimum size=4pt,inner sep=0pt]
                        \tikzstyle{edge} = [draw,line width=0.7pt,-,black]
                        \node[vertex] (v0) at (-15,5){};
                        \node[vertex] (v1) at (-5,5){};
                        \node[vertex] (v2) at (5,5){};
                        \node[vertex] (v3) at (15,5){};
                        \node[vertex] (v4) at (15,-5){};
                        \node[vertex] (v5) at (5,-5){};
                        \node[vertex] (v6) at (-5,-5){};
                        \node[vertex] (v7) at (-15,-5){};
                        \draw[edge]  (v0) to (v1);
                        \draw[edge]  (v1) to (v2);
                        \draw[edge]  (v2) to (v3);
                        \draw[edge]  (v3) to (v4);
                        \draw[edge]  (v4) to (v5);
                        \draw[edge]  (v5) to (v6);
                        \draw[edge]  (v6) to (v7);
                        \draw[edge]  (v7) to (v0);
                        \draw[edge]  (v1) to (v6);
                        \draw[edge]  (v2) to (v5);
                        \node[vertex] (v8) at (23,13){};
                        \node[vertex] (v9) at (31,18){};
                        \node[vertex] (v10) at (31,-18){};
                        \node[vertex] (v11) at (23,-13){};
                        \draw[edge]  (v3) to (v8);
                        \draw[edge]  (v8) to (v9);
                        \draw[edge]  (v9) to (v10);
                        \draw[edge]  (v10) to (v11);
                        \draw[edge]  (v11) to (v4);
                        \node[vertex] (v12) at (-23,13){};
                        \node[vertex] (v13) at (-31,18){};
                        \node[vertex] (v14) at (-31,-18){};
                        \node[vertex] (v15) at (-23,-13){};
                        \draw[edge]  (v0) to (v12);
                        \draw[edge]  (v12) to (v13);
                        \draw[edge]  (v13) to (v14);
                        \draw[edge]  (v14) to (v15);
                        \draw[edge]  (v15) to (v7);
                        \draw[edge]  (v12) to (v8);
                        \draw[edge]  (v15) to (v11);
                        \node[vertex] (v16) at (41,25){};
                        \node[vertex] (v19) at (41,-25){};
                        \draw[edge]  (v9) to (v16);
                        \draw[edge]  (v19) to (v10);
                        \node[vertex] (v20) at (-41,25){};
                        \node[vertex] (v23) at (-41,-25){};
                        \draw[edge]  (v13) to (v20);
                        \draw[edge]  (v23) to (v14);
                        \draw[edge]  (v16) to (v20);
                        \draw[edge]  (v19) to (v23);
                        \draw[edge]  (v20) to (v23);
                        \draw[edge]  (v16) to (v19);
                        \end{tikzpicture}}
                    \caption{Two examples of (4,6)-fullerene of lantern structure.}
            \label{lanternsample}
        \end{figure}
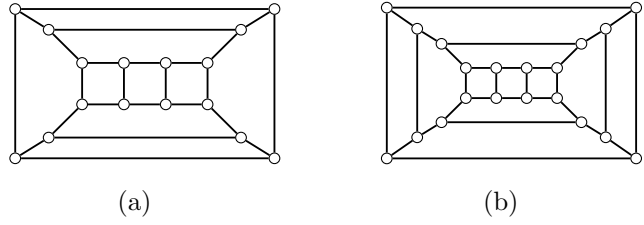

    \subsection{Classification theorem}

    A (4,6)-fullerene graph is said to be of {\em dispersive structure} if has neither three squares arranged in a line nor a square-cap. In the case of dispersive structure, every square face is adjacent to at most one square face.
        \begin{theorem}\label{category}
            Let $G$ be a (4,6)-fullerene graph. Then $G$ can be of one of the following~4 types:\\
                (i) a cube,\\
                (ii) a hexagonal prism,\\
                (iii) a tubular graph with at least one hexagon-layer,\\
                (iv) a (4,6)-fullerene graph of lantern structure, and\\
                (v) a (4,6)-fullerene graph of dispersive structure.
        \end{theorem}
        \begin{proof}
            If $h=0$, then $G$ is a cube. So we suppose $h \geqslant 2$. If $\zeta(G)=3$, then $G$ has a square-cap and is of tubular structure with at least one hexagon-layer. Further, we suppose $\zeta(G)=4$. If there are at least 4 squares arranged in a line, $G$ is a hexagonal prism by Proposition~\ref{lotogll}. If some exactly 3 squares are arranged in a line, then $G$ is of lantern structure by Proposition~\ref{lotogl}. For the remaining (4,6)-fullerene graph, it has neither three squares arranged in a line nor a square-cap. Hence it is of dispersive structure.
        \end{proof}

        As an application of Theorem~\ref{category}, we count 6-cycles.

        \begin{corollary}
            Let $G$ be a (4,6)-fullerene graph. Then\\
                (i) if $G$ is a cube, then the number of 6-cycles is 16,\\
                (ii) if $G=T_t\ (t\geqslant1)$, then the number of 6-cycles is $4t+7$,\\
                (iii) if $G\notin\mathcal{T}$ other than the cube, then the number of 6-cycles is $h+y$,  where $y$ is the number of dual-squares.
        \end{corollary}
        \begin{proof}
        (i) In a cube, there are in total 12 dual-squares. Also there are exactly four 6-cycles that are boundaries of square-caps. \\
        (ii)  $T_t\ (t\geqslant1)$ has $3t$  hexagonal faces, 6  dual-squares  and $(t+1)$  6-cycles that are boundaries of square-caps with or without hexagonal-layers. Lemma~\ref{6cycle2} demonstrates our result.\\
        (iii) If $G\notin\mathcal{T}$ other than the cube and has $y$ dual-squares, by Lemma~\ref{6cycle1} the number of 6-cycles is $h+y$.
        \end{proof}

\section{Number of low-order matchings}
    \label{sec_enum}
        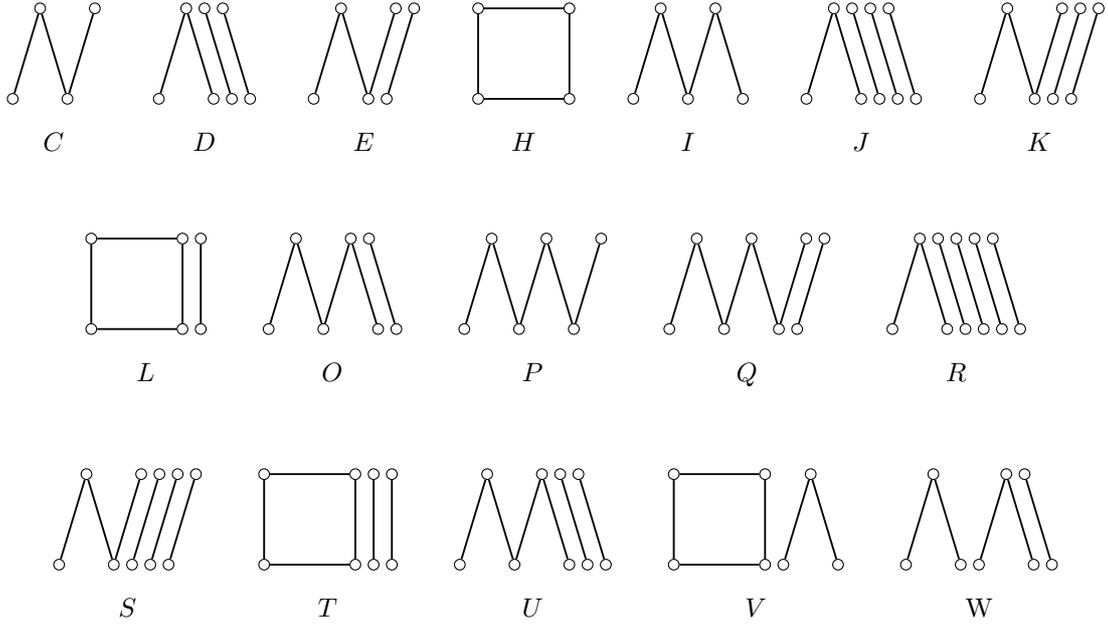
\begin{figure}[h]
            \centering
            \renewcommand{\thesubfigure}{}
                \subfigure[$C$]{
                    \begin{tikzpicture}[scale=1.2]
                        \tikzstyle{vertex}=[draw,circle,minimum size=4pt,inner sep=0pt]
                        \tikzstyle{edge} = [draw,line width=0.7pt,-,black]
                        \node[vertex] (v0) at (0,0)  {};
                        \node[vertex] (v1) at (0.3,1)  {};
                        \node[vertex] (v2) at (0.6,0)  {};
                        \node[vertex] (v3) at (0.9,1)  {};
                        \draw[edge] (v0)  -- (v1);
                        \draw[edge] (v1)  -- (v2);
                        \draw[edge] (v2)  -- (v3);
                    \end{tikzpicture}}
            \vspace{20pt}
            \hspace{1.5ex}
                \subfigure[$D$]{
                    \begin{tikzpicture}[scale=1.2]
                        \tikzstyle{vertex}=[draw,circle,minimum size=4pt,inner sep=0pt]
                        \tikzstyle{edge} = [draw,line width=0.7pt,-,black]
                        \node[vertex] (v0) at (0,0)  {};
                        \node[vertex] (v1) at (0.3,1) {};
                        \node[vertex] (v2) at (0.6,0) {};
                        \node[vertex] (v3) at (0.5,1) {};
                        \node[vertex] (v4) at (0.8,0) {};
                        \node[vertex] (v5) at (0.7,1) {};
                        \node[vertex] (v6) at (1.0,0) {};
                        \draw[edge] (v0)  -- (v1);
                        \draw[edge] (v1)  -- (v2);
                        \draw[edge] (v3)  -- (v4);
                        \draw[edge] (v5)  -- (v6);
                    \end{tikzpicture}}
            \hspace{1.5ex}
                \subfigure[$E$]{
                    \begin{tikzpicture}[scale=1.2]
                        \tikzstyle{vertex}=[draw,circle,minimum size=4pt,inner sep=0pt]
                        \tikzstyle{edge} = [draw,line width=0.7pt,-,black]
                        \node[vertex] (v0) at (0,0)  {};
                        \node[vertex] (v1) at (0.3,1)  {};
                        \node[vertex] (v2) at (0.6,0)  {};
                        \node[vertex] (v3) at (0.9,1)  {};
                        \node[vertex] (v4) at (0.8,0)  {};
                        \node[vertex] (v5) at (1.1,1)  {};
                        \draw[edge] (v0)  -- (v1);
                        \draw[edge] (v1)  -- (v2);
                        \draw[edge] (v2)  -- (v3);
                        \draw[edge] (v4)  -- (v5);
                    \end{tikzpicture}}
            \hspace{1.5ex}
                \subfigure[$H$]{
                    \begin{tikzpicture}[scale=1.2]
                        \tikzstyle{vertex}=[draw,circle,minimum size=4pt,inner sep=0pt]
                        \tikzstyle{edge} = [draw,line width=0.7pt,-,black]
                        \node[vertex] (v0) at (0,0)  {};
                        \node[vertex] (v1) at (0,1)  {};
                        \node[vertex] (v2) at (1,1)  {};
                        \node[vertex] (v3) at (1,0)  {};
                        \draw[edge] (v0)  -- (v1);
                        \draw[edge] (v1)  -- (v2);
                        \draw[edge] (v2)  -- (v3);
                        \draw[edge] (v3)  -- (v0);
                    \end{tikzpicture}}
            \hspace{1.5ex}
                \subfigure[$I$]{
                    \begin{tikzpicture}[scale=1.2]
                        \tikzstyle{vertex}=[draw,circle,minimum size=4pt,inner sep=0pt]
                        \tikzstyle{edge} = [draw,line width=0.7pt,-,black]
                        \node[vertex] (v0) at (0,0)  {};
                        \node[vertex] (v1) at (0.3,1)  {};
                        \node[vertex] (v2) at (0.6,0)  {};
                        \node[vertex] (v3) at (0.9,1)  {};
                        \node[vertex] (v4) at (1.2,0)  {};
                        \draw[edge] (v0)  -- (v1);
                        \draw[edge] (v1)  -- (v2);
                        \draw[edge] (v2)  -- (v3);
                        \draw[edge] (v3)  -- (v4);
                    \end{tikzpicture}}
            \hspace{1.5ex}
                \subfigure[$J$]{
                    \begin{tikzpicture}[scale=1.2]
                        \tikzstyle{vertex}=[draw,circle,minimum size=4pt,inner sep=0pt]
                        \tikzstyle{edge} = [draw,line width=0.7pt,-,black]
                        \node[vertex] (v0) at (0,0) {};
                        \node[vertex] (v1) at (0.3,1) {};
                        \node[vertex] (v2) at (0.6,0) {};
                        \node[vertex] (v3) at (0.5,1) {};
                        \node[vertex] (v4) at (0.8,0) {};
                        \node[vertex] (v5) at (0.7,1) {};
                        \node[vertex] (v6) at (1.0,0) {};
                        \node[vertex] (v7) at (0.9,1) {};
                        \node[vertex] (v8) at (1.2,0) {};
                        \draw[edge] (v0)  -- (v1);
                        \draw[edge] (v1)  -- (v2);
                        \draw[edge] (v3)  -- (v4);
                        \draw[edge] (v5)  -- (v6);
                        \draw[edge] (v7)  -- (v8);
                    \end{tikzpicture}}
            \hspace{1.5ex}
                \subfigure[$K$]{
                    \begin{tikzpicture}[scale=1.2]
                        \tikzstyle{vertex}=[draw,circle,minimum size=4pt,inner sep=0pt]
                        \tikzstyle{edge} = [draw,line width=0.7pt,-,black]
                        \node[vertex] (v0) at (0,0)  {};
                        \node[vertex] (v1) at (0.3,1)  {};
                        \node[vertex] (v2) at (0.6,0)  {};
                        \node[vertex] (v3) at (0.9,1)  {};
                        \node[vertex] (v4) at (0.8,0)  {};
                        \node[vertex] (v5) at (1.1,1)  {};
                        \node[vertex] (v6) at (1.0,0)  {};
                        \node[vertex] (v7) at (1.3,1)  {};
                        \draw[edge] (v0)  -- (v1);
                        \draw[edge] (v1)  -- (v2);
                        \draw[edge] (v2)  -- (v3);
                        \draw[edge] (v4)  -- (v5);
                        \draw[edge] (v6)  -- (v7);
                    \end{tikzpicture}}\\
                \subfigure[$L$]{
                    \begin{tikzpicture}[scale=1.2]
                        \tikzstyle{vertex}=[draw,circle,minimum size=4pt,inner sep=0pt]
                        \tikzstyle{edge} = [draw,line width=0.7pt,-,black]
                        \node[vertex] (v0) at (0,0)  {};
                        \node[vertex] (v1) at (0,1)  {};
                        \node[vertex] (v2) at (1,1)  {};
                        \node[vertex] (v3) at (1,0)  {};
                        \node[vertex] (v4) at (1.2,0)  {};
                        \node[vertex] (v5) at (1.2,1)  {};
                        \draw[edge] (v0)  -- (v1);
                        \draw[edge] (v1)  -- (v2);
                        \draw[edge] (v2)  -- (v3);
                        \draw[edge] (v3)  -- (v0);
                        \draw[edge] (v4)  -- (v5);
                    \end{tikzpicture}}
            \vspace{20pt}
            \hspace{1.8ex}
                \subfigure[$O$]{
                    \begin{tikzpicture}[scale=1.2]
                        \tikzstyle{vertex}=[draw,circle,minimum size=4pt,inner sep=0pt]
                        \tikzstyle{edge} = [draw,line width=0.7pt,-,black]
                        \node[vertex] (v0) at (0,0)  {};
                        \node[vertex] (v1) at (0.3,1)  {};
                        \node[vertex] (v2) at (0.6,0)  {};
                        \node[vertex] (v3) at (0.9,1)  {};
                        \node[vertex] (v4) at (1.2,0)  {};
                        \node[vertex] (v5) at (1.1,1) {};
                        \node[vertex] (v6) at (1.4,0) {};
                        \draw[edge] (v0)  -- (v1);
                        \draw[edge] (v1)  -- (v2);
                        \draw[edge] (v2)  -- (v3);
                        \draw[edge] (v3)  -- (v4);
                        \draw[edge] (v5)  -- (v6);
                    \end{tikzpicture}}
            \hspace{1.8ex}
                 \subfigure[$P$]{
                    \begin{tikzpicture}[scale=1.2]
                        \tikzstyle{vertex}=[draw,circle,minimum size=4pt,inner sep=0pt]
                        \tikzstyle{edge} = [draw,line width=0.7pt,-,black]
                        \node[vertex] (v0) at (0,0)  {};
                        \node[vertex] (v1) at (0.3,1)  {};
                        \node[vertex] (v2) at (0.6,0) {};
                        \node[vertex] (v3) at (0.9,1) {};
                        \node[vertex] (v4) at (1.2,0) {};
                        \node[vertex] (v5) at (1.5,1) {};
                        \draw[edge] (v0)  -- (v1);
                        \draw[edge] (v1)  -- (v2);
                        \draw[edge] (v2)  -- (v3);
                        \draw[edge] (v3)  -- (v4);
                        \draw[edge] (v4)  -- (v5);
                    \end{tikzpicture}}
            \hspace{1.8ex}
                 \subfigure[$Q$]{
                    \begin{tikzpicture}[scale=1.2]
                        \tikzstyle{vertex}=[draw,circle,minimum size=4pt,inner sep=0pt]
                        \tikzstyle{edge} = [draw,line width=0.7pt,-,black]
                        \node[vertex] (v0) at (0,0)  {};
                        \node[vertex] (v1) at (0.3,1)  {};
                        \node[vertex] (v2) at (0.6,0) {};
                        \node[vertex] (v3) at (0.9,1) {};
                        \node[vertex] (v4) at (1.2,0) {};
                        \node[vertex] (v5) at (1.5,1) {};
                        \node[vertex] (v6) at (1.4,0) {};
                        \node[vertex] (v7) at (1.7,1)  {};
                        \draw[edge] (v0)  -- (v1);
                        \draw[edge] (v1)  -- (v2);
                        \draw[edge] (v2)  -- (v3);
                        \draw[edge] (v3)  -- (v4);
                        \draw[edge] (v4)  -- (v5);
                        \draw[edge] (v6)  -- (v7);
                    \end{tikzpicture}}
            \hspace{1.8ex}
                 \subfigure[$R$]{
                    \begin{tikzpicture}[scale=1.2]
                        \tikzstyle{vertex}=[draw,circle,minimum size=4pt,inner sep=0pt]
                        \tikzstyle{edge} = [draw,line width=0.7pt,-,black]
                        \node[vertex] (v0) at (0,0) {};
                        \node[vertex] (v1) at (0.3,1) {};
                        \node[vertex] (v2) at (0.6,0) {};
                        \node[vertex] (v3) at (0.5,1) {};
                        \node[vertex] (v4) at (0.8,0) {};
                        \node[vertex] (v5) at (0.7,1) {};
                        \node[vertex] (v6) at (1.0,0) {};
                        \node[vertex] (v7) at (0.9,1) {};
                        \node[vertex] (v8) at (1.2,0) {};
                        \node[vertex] (v9) at (1.1,1) {};
                        \node[vertex] (v10) at (1.4,0){};
                        \draw[edge] (v0)  -- (v1);
                        \draw[edge] (v1)  -- (v2);
                        \draw[edge] (v3)  -- (v4);
                        \draw[edge] (v5)  -- (v6);
                        \draw[edge] (v7)  -- (v8);
                        \draw[edge] (v9)  -- (v10);
                    \end{tikzpicture}}\\
                 \subfigure[$S$]{
                    \begin{tikzpicture}[scale=1.2]
                        \tikzstyle{vertex}=[draw,circle,minimum size=4pt,inner sep=0pt]
                        \tikzstyle{edge} = [draw,line width=0.7pt,-,black]
                        \node[vertex] (v0) at (0,0) {};
                        \node[vertex] (v1) at (0.3,1) {};
                        \node[vertex] (v2) at (0.6,0) {};
                        \node[vertex] (v3) at (0.9,1) {};
                        \node[vertex] (v4) at (0.8,0) {};
                        \node[vertex] (v5) at (1.1,1) {};
                        \node[vertex] (v6) at (1.0,0) {};
                        \node[vertex] (v7) at (1.3,1) {};
                        \node[vertex] (v8) at (1.2,0) {};
                        \node[vertex] (v9) at (1.5,1) {};
                        \draw[edge] (v0)  -- (v1);
                        \draw[edge] (v1)  -- (v2);
                        \draw[edge] (v2)  -- (v3);
                        \draw[edge] (v4)  -- (v5);
                        \draw[edge] (v6)  -- (v7);
                        \draw[edge] (v8)  -- (v9);
                    \end{tikzpicture}}
            \hspace{1.8ex}
                \subfigure[$T$]{
                    \begin{tikzpicture}[scale=1.2]
                        \tikzstyle{vertex}=[draw,circle,minimum size=4pt,inner sep=0pt]
                        \tikzstyle{edge} = [draw,line width=0.7pt,-,black]
                        \node[vertex] (v0) at (0,0)  {};
                        \node[vertex] (v1) at (0,1)  {};
                        \node[vertex] (v2) at (1,1)  {};
                        \node[vertex] (v3) at (1,0)  {};
                        \node[vertex] (v4) at (1.2,0)  {};
                        \node[vertex] (v5) at (1.2,1)  {};
                        \node[vertex] (v6) at (1.4,0)  {};
                        \node[vertex] (v7) at (1.4,1)  {};
                        \draw[edge] (v0)  -- (v1);
                        \draw[edge] (v1)  -- (v2);
                        \draw[edge] (v2)  -- (v3);
                        \draw[edge] (v3)  -- (v0);
                        \draw[edge] (v4)  -- (v5);
                        \draw[edge] (v6)  -- (v7);
                    \end{tikzpicture}}
            \hspace{1.8ex}
                 \subfigure[$U$]{
                    \begin{tikzpicture}[scale=1.2]
                        \tikzstyle{vertex}=[draw,circle,minimum size=4pt,inner sep=0pt]
                        \tikzstyle{edge} = [draw,line width=0.7pt,-,black]
                        \node[vertex] (v0) at (0,0)  {};
                        \node[vertex] (v1) at (0.3,1)  {};
                        \node[vertex] (v2) at (0.6,0)  {};
                        \node[vertex] (v3) at (0.9,1)  {};
                        \node[vertex] (v4) at (1.2,0)  {};
                        \node[vertex] (v5) at (1.1,1)  {};
                        \node[vertex] (v6) at (1.4,0)  {};
                        \node[vertex] (v7) at (1.3,1)  {};
                        \node[vertex] (v8) at (1.6,0)  {};
                        \draw[edge] (v0)  -- (v1);
                        \draw[edge] (v1)  -- (v2);
                        \draw[edge] (v2)  -- (v3);
                        \draw[edge] (v3)  -- (v4);
                        \draw[edge] (v5)  -- (v6);
                        \draw[edge] (v7)  -- (v8);
                    \end{tikzpicture}}
            \hspace{1.8ex}
                 \subfigure[$V$]{
                    \begin{tikzpicture}[scale=1.2]
                            \tikzstyle{vertex}=[draw,circle,minimum size=4pt,inner sep=0pt]
                            \tikzstyle{edge} = [draw,line width=0.7pt,-,black]
                            \node[vertex] (v0) at (0,0) {};
                            \node[vertex] (v1) at (0.3,1) {};
                            \node[vertex] (v2) at (0.6,0){};
                            \draw[edge] (v0)  -- (v1);
                            \draw[edge] (v1)  -- (v2);
                            \node[vertex] (v3) at (-1.2,0) {};
                            \node[vertex] (v4) at (-0.2,0) {};
                            \node[vertex] (v5) at (-0.2,1) {};
                            \node[vertex] (v6) at (-1.2,1) {};
                            \draw[edge] (v3)  -- (v4);
                            \draw[edge] (v4)  -- (v5);
                            \draw[edge] (v5)  -- (v6);
                            \draw[edge] (v6)  -- (v3);
                        \end{tikzpicture}}
            \hspace{1.8ex}
                 \subfigure[W]{
                    \begin{tikzpicture}[scale=1.2]
                        \tikzstyle{vertex}=[draw,circle,minimum size=4pt,inner sep=0pt]
                        \tikzstyle{edge} = [draw,line width=0.7pt,-,black]
                        \node[vertex] (v0) at (0,0)  {};
                        \node[vertex] (v1) at (0.3,1)  {};
                        \node[vertex] (v2) at (0.6,0)  {};
                        \node[vertex] (v22) at (0.8,0)  {};
                        \node[vertex] (v3) at (1.1,1)  {};
                        \node[vertex] (v4) at (1.4,0)  {};
                        \node[vertex] (v5) at (1.3,1) {};
                        \node[vertex] (v6) at (1.6,0) {};
                        \draw[edge] (v0)  -- (v1);
                        \draw[edge] (v1)  -- (v2);
                        \draw[edge] (v22)  -- (v3);
                        \draw[edge] (v3)  -- (v4);
                        \draw[edge] (v5)  -- (v6);
                    \end{tikzpicture}}
            \caption{Some possible subgraphs of a (4,6)-fullerene graph.}
            \label{collection}
        \end{figure}
        Let $G$ be a (4,6)-fullerene graph. We will consider several small subgraphs of (4,6)-fullerene graphs. Those of interest to us are shown in Figure~\ref{collection} and will be denoted by $C, D, E, H,\dots, W$ as shown in the figure. For a subgraph $S$ of $G$, let $N(S)$ be the number of subgraphs of $G$ that are isomorphic to $S$.
        \begin{lemma}\label{pre0}
            Let $G$ be a (4,6)-fullerene graph. Then\\
                (i)\cite{behmaram13} $N(C)=4m$, $N(I)=8m-24$;\\
                (ii)\cite{li14} $N(O)=72h^2+240h+120$, $N(L)=18h+24$.
        \end{lemma}
        \begin{lemma}\label{pre}
            Let $G$ be a (4,6)-fullerene graph. Then $N(P)=48h+120$, $N(T)=27h^2+27h+12$ and $N(V)=36h+24$.
        \end{lemma}
        \begin{proof}
            To count $N(P)$, we choose one path of length 3 first and then choose an edge from the two endpoints  of the path independently. Noting that the chosen 3-length path may lie on a square,  we have
            \begin{displaymath}
                N(P)=(N(C)-6\times4)\times2\times2+6\times4=48h+120.
            \end{displaymath}
            \par
            We now consider $N(T)$. We choose a square and then two edges not incident to the chosen square. Then we can get a subgraph $T$ or $V$. $N(V)$ can be calculated through the following formula
            \begin{displaymath}
                N(V)=6\times(n-8)\times3+6\times4=36h+24.
            \end{displaymath}
            Thus, we have
            \begin{displaymath}
                N(T)=6\times\binom{m-8}{2}-N(V)=27h^2+27h+12.
            \end{displaymath}
        \end{proof}
        \begin{remark}
            The number $N(P)$ is not consistent with that in~\cite{behmaram13}. Behmaram et al. obtained $N(P)$ by choosing an edge at first and then extending the path to both directions. With this in mind~\cite{behmaram13} gives a formula $N(P)= m\times4\times4-48$, where~48 deals with some cases in which the original edge is located in a square and the resulted graph after extension is a square with a hanging edge. However, they did not notice whether the original edge is in a square and the extension process would be restricted. Thus the formula is not correct. Since Li et al.~\cite{li14} obtain the enumeration of 5-matching by applying the $N(P)$ from~\cite{behmaram13}, their formula  is also wrong. Moreover, a similar extension restriction leads to a mistake in the calculation of $N(W)$ in~\cite{li14}.
        \end{remark}
        \begin{lemma}\cite{behmaram13,li14}\label{mm123}
            Let $G$ be a (4,6)-fullerene graph. Then we have
            \begin{equation}\begin{split}\notag
                M(G,1)&=3h+12,\\
                M(G,2)&=\dfrac{9}{2}h^2+\dfrac{57}{2}h+42,\\
                M(G,3)&=\dfrac{9}{2}h^3+\dfrac{63}{2}h^2+65h+44.
            \end{split}\end{equation}
        \end{lemma}
    \subsection{Recurrence relations}
       \begin{lemma}\label{receq}
                Let $G$ be a (4,6)-fullerene graph. We have\\                    (i)
                        \begin{subequations}\label{m6pr}
                            \begin{align}
                                M(G,5)\times(m-5)&=6\times M(G,6)+2\times N(R)+N(S), \label{m6pr1}\\
                                M(G,5)\times10\times2&=2\times N(R)+2\times N(S),\label{m6pr2}\\
                                M(G,4)\times 4\times 4 &=N(S)+4\times N(T)+2\times N(U)+N(Q), \label{m6pr3}\\
                                N(K)\times2\times2&=2\times N(U)+2\times N(Q)+8\times N(T); \label{m6pr4}
                            \end{align}
                        \end{subequations}
                    (ii)
                        \begin{subequations}\label{m5pr}
                            \begin{align}
                                M(G,4)\times(m-4)&= 5\times M(G,5)+2\times N(J)+N(K),\label{m5pr1}\\
                                M(G,4)\times8\times2&=2\times N(J)+2\times N(K),\label{m5pr2}\\
                                M(G,3)\times 3\times 4&=N(K)+4\times N(L)+2\times N(O)+N(P); \label{m5pr3}
                            \end{align}
                        \end{subequations}
                    (iii)
                        \begin{subequations}\label{m4pr}
                            \begin{align}
                                M(G,3)\times(m-3)&=4\times M(G,4)+2\times N(D)+N(E), \label{m4pr1}\\
                                M(G,3)\times6\times2&=2\times N(D)+2\times N(E),\label{m4pr2}\\
                                M(G,2)\times 2\times 4&= N(E)+4\times N(H)+2\times N(I).\label{m4pr3}
                            \end{align}
                        \end{subequations}
            \end{lemma}
            \begin{proof}
                We only give the proof of equation system~\eqref{m6pr}. The  proofs of systems~\eqref{m5pr} and~\eqref{m4pr} are essentially the same and omitted here.
                \par
                To prove \eqref{m6pr1},  we choose a 5-matching first and then add an arbitrary edge. We use an ordered pair to represent such a choice: the first coordinate indicates the chosen 5-matching and the second one stands for the chosen edge. By the elementary counting principle, the number of such ordered pairs (which equals the number of ways to do the above process) is
                \begin{displaymath}
                    M(G,5)\times (m-5).
                \end{displaymath}
                \par
                Hereinafter, we say a subgraph of $G$ and an ordered pair described above are correlated if the process represented by the pair leads to the subgraph. It is obvious that after the above process, the resulted subgraph of $G$ could only be a 6-matching, a subgraph $R$ or a subgraph $S$. Each 6-matching correlates~6 ordered pairs. Similarly, for each subgraph $R$ (respectively, $S$), there are two (respectively, one) correlated pairs. Thus, the number of representation pairs is
                \begin{displaymath}
                    6\times M(G,6)+2\times N(R)+N(S).
                \end{displaymath}
                \par
                Hence \eqref{m6pr1} is proved and we take in consideration a second process for the proof of~\eqref{m6pr2}. First, we choose a 5-matching and fix one of its vertices, then we choose one edge incident with the chosen vertex that is not in the 5-matching. Similarly, we define ordered triples to represent the above process: the first coordinate indicates the 5-matching chosen, while the second one is for the fixed vertex and the last one for the chosen edge. The number of such ordered triples is
                \begin{displaymath}
                    M(G,5)\times 10\times2.
                \end{displaymath}
                \par
                The resulted subgraph of $G$ after such a process could only be a subgraph $R$, or a subgraph $S$. Each subgraph $R$ correlates to two representation triples since we have 2 ways to determine the initially chosen 5-matching for a given a subgraph $R$. As for each subgraph $S$ we have two ways to determine the fixed vertex, each subgraph $S$ correlates to two triples. Thus, the number of representation triples equals
                \begin{displaymath}
                    2\times N(R)+2\times N(S).
                \end{displaymath}
                We have proved \eqref{m6pr2}.
                \par
                To prove \eqref{m6pr3}, we choose a 4-matching first and fix one  edge of this 4-matching, then we add one edge for each end vertex of the fixed edge. Again, the number of representation ordered triples is
                \begin{displaymath}
                    M(G,4)\times 4\times 4.
                \end{displaymath}
                \par
                Once we execute the above process, the resulted subgraph of $G$ should be  $S$,  $T$,  $U$, or  $Q$. Each subgraph $T$ correlates to~4 triples, since for a given $T$, we have 2 ways to determine the initial 4-matching and further if the 4-matching is determined we have 2 ways to determine the fixed edge. Similarly, each subgraph $U$ (respectively, $S$ and $Q$) correlates to two (respectively, one and one) ordered triples. Thus, the number of representation triples equals
                \begin{displaymath}
                    4\times N(T)+2\times N(U)+N(Q)+N(S).
                \end{displaymath}
                \par
                Therefore \eqref{m6pr3} is proved and we turn to~\eqref{m6pr4}. To this end, we choose a subgraph $K$ first and fix one end-vertex of the 3-length path in $K$, then we add one edge incident with the fixed vertex. And we use ordered triples to represent the process as above: the first coordinate indicates the subgraph $K$ chosen,  the second coordinate stands for the fixed vertex and the third one means the added edge. Then the number of representation triples is $N(K)\times2\times2$. The resulted graph of such a process is either a subgraph $U$, $Q$ or $T$. For each $U$ or $Q$, there are two triples correlated since we have two ways to determine the initial $K$. For each $T$, it is easy to check that there are 8 correlated triples. Hence~\eqref{m6pr4} holds.
            \end{proof}
By~\eqref{m5pr3}, we have
            \begin{corollary}\label{coroK}
                $
                    N(K)=54h^3+234h^2+180h+72.
                $
            \end{corollary}

            \begin{theorem}\label{res}
                For  a (4,6)-fullerene graph $G$, we have
                \begin{subequations}
                    \begin{align}
                        M(G,4)&=\dfrac{27}{8}h^4+\dfrac{81}{4}h^3+\dfrac{273}{8}h^2+\dfrac{117}{4}h+9,\label{mm4}\\
                        M(G,5)&=\dfrac{81}{40}h^5+\dfrac{27}{4}h^4-\dfrac{9}{8}h^3+\dfrac{39}{4}h^2-\dfrac{27}{5}h,\label{mm5}\\
                        M(G,6)&= \dfrac{81}{80}h^6-\dfrac{81}{80}h^5-\dfrac{99}{16}h^4+\dfrac{405}{16}h^3-\dfrac{2833}{40}h^2-\dfrac{123}{10}h-16+\dfrac{N(Q)}{6}.\label{mm6rec}
                    \end{align}
                \end{subequations}
            \end{theorem}
            \begin{proof}
                Calculating $\eqref{m4pr1}-\eqref{m4pr2}+\eqref{m4pr3}$, $\eqref{m5pr1}-\eqref{m5pr2}+\eqref{m5pr3}$ and $\eqref{m6pr1}-\eqref{m6pr2}+\eqref{m6pr3}-\eqref{m6pr4}$, respectively, we get
                \begin{displaymath}\begin{split}
                    4M(G,4)+4N(H)+2N(I)&=(m-15)M(G,3)+8M(G,2);\\
                    5M(G,5)+4N(L)+2N(O)+N(P)&=(m-20)M(G,4)+12M(G,3);\\
                    6M(G,6)+4N(K)-4N(T)&=(m-25)M(G,5)+16M(G,4)+N(Q).
                \end{split}\end{displaymath}
                Equivalently,
                \begin{displaymath}\begin{split}
                    M(G,4)&=\dfrac{m-15}{4}M(G,3)+2M(G,2)-N(H)-\dfrac{1}{2}N(I);\\
                    M(G,5)&=\dfrac{m-20}{5}M(G,4)+\dfrac{12}{5}M(G,3)-\dfrac{4}{5}N(L)-\dfrac{2}{5}N(O)-\dfrac{1}{5}N(P);\\
                    M(G,6)&=\dfrac{m-25}{6}M(G,5)+\dfrac{8}{3}M(G,4)-\dfrac{2}{3}N(K)+\dfrac{2}{3}N(T)+\dfrac{1}{6}N(Q).
                \end{split}\end{displaymath}
                    That is a series of recurrence relations. Our theorem follows immediately.
            \end{proof}
            Theorem~\ref{res} has already given the expressions of $M(G,4)$ and $M(G,5)$, and the enumeration of 6-matchings is now reduced to the calculation of $N(Q)$. In addition, we have the following corollary from  Lemma~\ref{receq}.
            \begin{corollary}\label{coroEX}
                Let $G$ be a (4,6)-fullerene graph. We have
                \begin{align}
                    N(E)&=36h^2+180h+168;\label{coroEX1}\\
                    N(D)&=27h^3+153h^2+210h+96;\label{coroEX2}\\
                    N(J)&=27h^4+108h^3+39h^2+54h;\label{coroEX3}\\
                    N(U)&=108h^3+360h^2+252h+96-N(Q);\label{coroEX4}\\
                    N(S)&=54h^4+108h^3-282h^2-144h-96+N(Q);\label{coroEX5}\\
                    N(R)&=\dfrac{81}{4}h^5+\dfrac{27}{2}h^4-\dfrac{477}{4}h^3+\dfrac{759}{2}h^2+90h+96-N(Q).\label{coroEX6}
                \end{align}
            \end{corollary}
            \begin{proof}
                ~\eqref{m4pr3} gives $N(E)=8M(G,2)-4N(H)-2N(I)$, so we have~\eqref{coroEX1}.
                ~\eqref{m4pr2} gives $N(D)=6M(G,3)-N(E)$, so~\eqref{coroEX2} follows.
                By~\eqref{m5pr2}, $N(J)=8M(G,4)-N(K)$, which leads to~\eqref{coroEX3}.
                By~\eqref{m6pr4}, $N(U)=2N(K)-4N(T)-N(Q)$, which leads to ~\eqref{coroEX4}.
                Substracting~\eqref{m6pr3} with~\eqref{m6pr4}, we get $N(S)=16M(G,4)-4N(K)+4N(T)+N(Q)$, so~\eqref{coroEX5} holds.
                It follows from~\eqref{m6pr2} that $N(R)=10M(G,5)-N(S)$, and~\eqref{coroEX6} is proved.
            \end{proof}

    \subsection{Number of 6-matchings}
        In the following, we calculate $N(Q)$ for (4,6)-fullerene graphs according to their different structures.
            \begin{lemma}\label{qtub}
                Let $G=T_t\ (t\geqslant1)$ be a tubular (4,6)-fullerene graph. Then
                \begin{displaymath}
                    N(Q)=144h^2+320h+42.
                \end{displaymath}
            \end{lemma}
            \begin{proof}
                To calculate $N(Q)$, we choose a path of length 5 (a subgraph $P$) first and then choose an edge disjoint with the path. We should notice that a 5-length path may be in a 6-cycle and some 3 consecutive edges  of the path may be located in a 4-cycle. \\
                \begin{figure*}[h]
                    \centering
                    \subfigure{
                        \begin{tikzpicture}[scale=0.09]
                            \tikzstyle{vertex}=[draw,circle,minimum size=4pt,inner sep=0pt]
                            \tikzstyle{edge} = [draw,line width=0.7pt,-,black]
                            \tikzstyle{selected edge} = [draw,line width=1.5pt,-,black]
                            \node[vertex] (v0) at (0,0)  {};
                            \node[vertex] (v1) at (5,0)  {};
                            \node[vertex] (v2) at (2.5,4.32)   {};
                            \node[vertex] (v3) at (-2.5,4.32)   {};
                            \node[vertex] (v4) at (-5,0)  {};
                            \node[vertex] (v5) at (-2.5,-4.32)   {};
                            \node[vertex] (v6) at (2.5,-4.32)  {};
                            \node[vertex] (v7) at (10,0)  {};
                            \node[vertex] (v8) at (5,8.64)   {};
                            \node[vertex] (v9) at (-5,8.64)   {};
                            \node[vertex] (v10) at (-10,0)  {};
                            \node[vertex] (v11) at (-5,-8.64)   {};
                            \node[vertex] (v12) at (5,-8.64)  {};
                            \draw[edge] (v1)  -- (v2);
                            \draw[edge] (v2)  -- (v3);
                            \draw[edge] (v3)  -- (v4);
                            \draw[edge] (v4)  -- (v5);
                            \draw[edge] (v5)  -- (v6);
                            \draw[selected edge] (v6)  -- (v1);
                            \draw[selected edge] (v0)  -- (v1);
                            \draw[edge] (v0)  -- (v3);
                            \draw[selected edge] (v0)  -- (v5);
                            \draw[edge] (v7)  -- (v8);
                            \draw[edge] (v8)  -- (v9);
                            \draw[edge] (v9)  -- (v10);
                            \draw[edge] (v10)  -- (v11);
                            \draw[selected edge] (v11)  -- (v12);
                            \draw[edge] (v12)  -- (v7);
                            \draw[edge] (v2)  -- (v8);
                            \draw[edge] (v4)  -- (v10);
                            \draw[selected edge] (v6)  -- (v12);
                        \end{tikzpicture}}
                    \subfigure{
                        \begin{tikzpicture}[scale=0.09]
                            \tikzstyle{vertex}=[draw,circle,minimum size=4pt,inner sep=0pt]
                            \tikzstyle{edge} = [draw,line width=0.7pt,-,black]
                            \tikzstyle{selected edge} = [draw,line width=1.5pt,-,black]
                            \node[vertex] (v0) at (0,0)  {};
                            \node[vertex] (v1) at (5,0)  {};
                            \node[vertex] (v2) at (2.5,4.32)   {};
                            \node[vertex] (v3) at (-2.5,4.32)   {};
                            \node[vertex] (v4) at (-5,0)  {};
                            \node[vertex] (v5) at (-2.5,-4.32)   {};
                            \node[vertex] (v6) at (2.5,-4.32)  {};
                            \node[vertex] (v7) at (10,0)  {};
                            \node[vertex] (v8) at (5,8.64)   {};
                            \node[vertex] (v9) at (-5,8.64)   {};
                            \node[vertex] (v10) at (-10,0)  {};
                            \node[vertex] (v11) at (-5,-8.64)   {};
                            \node[vertex] (v12) at (5,-8.64)  {};
                            \draw[edge] (v1)  -- (v2);
                            \draw[edge] (v2)  -- (v3);
                            \draw[edge] (v3)  -- (v4);
                            \draw[edge] (v4)  -- (v5);
                            \draw[edge] (v5)  -- (v6);
                            \draw[selected edge] (v6)  -- (v1);
                            \draw[selected edge] (v0)  -- (v1);
                            \draw[edge] (v0)  -- (v3);
                            \draw[selected edge] (v0)  -- (v5);
                            \draw[edge] (v7)  -- (v8);
                            \draw[edge] (v8)  -- (v9);
                            \draw[edge] (v9)  -- (v10);
                            \draw[edge] (v10)  -- (v11);
                            \draw[edge] (v11)  -- (v12);
                            \draw[selected edge] (v12)  -- (v7);
                            \draw[edge] (v2)  -- (v8);
                            \draw[edge] (v4)  -- (v10);
                            \draw[selected edge] (v6)  -- (v12);
                        \end{tikzpicture}}
                    \subfigure{
                        \begin{tikzpicture}[scale=0.09]
                            \tikzstyle{vertex}=[draw,circle,minimum size=4pt,inner sep=0pt]
                            \tikzstyle{edge} = [draw,line width=0.7pt,-,black]
                            \tikzstyle{selected edge} = [draw,line width=1.5pt,-,black]
                            \node[vertex] (v0) at (0,0)  {};
                            \node[vertex] (v1) at (5,0)  {};
                            \node[vertex] (v2) at (2.5,4.32)   {};
                            \node[vertex] (v3) at (-2.5,4.32)   {};
                            \node[vertex] (v4) at (-5,0)  {};
                            \node[vertex] (v5) at (-2.5,-4.32)   {};
                            \node[vertex] (v6) at (2.5,-4.32)  {};
                            \node[vertex] (v7) at (10,0)  {};
                            \node[vertex] (v8) at (5,8.64)   {};
                            \node[vertex] (v9) at (-5,8.64)   {};
                            \node[vertex] (v10) at (-10,0)  {};
                            \node[vertex] (v11) at (-5,-8.64)   {};
                            \node[vertex] (v12) at (5,-8.64)  {};
                            \draw[selected edge] (v1)  -- (v2);
                            \draw[edge] (v2)  -- (v3);
                            \draw[edge] (v3)  -- (v4);
                            \draw[edge] (v4)  -- (v5);
                            \draw[selected edge] (v5)  -- (v6);
                            \draw[edge] (v6)  -- (v1);
                            \draw[selected edge] (v0)  -- (v1);
                            \draw[edge] (v0)  -- (v3);
                            \draw[selected edge] (v0)  -- (v5);
                            \draw[edge] (v7)  -- (v8);
                            \draw[edge] (v8)  -- (v9);
                            \draw[edge] (v9)  -- (v10);
                            \draw[edge] (v10)  -- (v11);
                            \draw[edge] (v11)  -- (v12);
                            \draw[edge] (v12)  -- (v7);
                            \draw[selected edge] (v2)  -- (v8);
                            \draw[edge] (v4)  -- (v10);
                            \draw[edge] (v6)  -- (v12);
                        \end{tikzpicture}}
                    \subfigure{
                        \begin{tikzpicture}[scale=0.09]
                            \tikzstyle{vertex}=[draw,circle,minimum size=4pt,inner sep=0pt]
                            \tikzstyle{edge} = [draw,line width=0.7pt,-,black]
                            \tikzstyle{selected edge} = [draw,line width=1.5pt,-,black]
                            \node[vertex] (v0) at (0,0)  {};
                            \node[vertex] (v1) at (5,0)  {};
                            \node[vertex] (v2) at (2.5,4.32)   {};
                            \node[vertex] (v3) at (-2.5,4.32)   {};
                            \node[vertex] (v4) at (-5,0)  {};
                            \node[vertex] (v5) at (-2.5,-4.32)   {};
                            \node[vertex] (v6) at (2.5,-4.32)  {};
                            \node[vertex] (v7) at (10,0)  {};
                            \node[vertex] (v8) at (5,8.64)   {};
                            \node[vertex] (v9) at (-5,8.64)   {};
                            \node[vertex] (v10) at (-10,0)  {};
                            \node[vertex] (v11) at (-5,-8.64)   {};
                            \node[vertex] (v12) at (5,-8.64)  {};
                            \draw[edge] (v1)  -- (v2);
                            \draw[edge] (v2)  -- (v3);
                            \draw[edge] (v3)  -- (v4);
                            \draw[selected edge] (v4)  -- (v5);
                            \draw[edge] (v5)  -- (v6);
                            \draw[selected edge] (v6)  -- (v1);
                            \draw[selected edge] (v0)  -- (v1);
                            \draw[edge] (v0)  -- (v3);
                            \draw[selected edge] (v0)  -- (v5);
                            \draw[edge] (v7)  -- (v8);
                            \draw[edge] (v8)  -- (v9);
                            \draw[edge] (v9)  -- (v10);
                            \draw[edge] (v10)  -- (v11);
                            \draw[edge] (v11)  -- (v12);
                            \draw[edge] (v12)  -- (v7);
                            \draw[edge] (v2)  -- (v8);
                            \draw[selected edge] (v4)  -- (v10);
                            \draw[edge] (v6)  -- (v12);
                        \end{tikzpicture}}\\
                    \subfigure{
                        \begin{tikzpicture}[scale=0.09]
                            \tikzstyle{vertex}=[draw,circle,minimum size=4pt,inner sep=0pt]
                            \tikzstyle{edge} = [draw,line width=0.7pt,-,black]
                            \tikzstyle{selected edge} = [draw,line width=1.5pt,-,black]
                            \node[vertex] (v0) at (0,0)  {};
                            \node[vertex] (v1) at (5,0)  {};
                            \node[vertex] (v2) at (2.5,4.32)   {};
                            \node[vertex] (v3) at (-2.5,4.32)   {};
                            \node[vertex] (v4) at (-5,0)  {};
                            \node[vertex] (v5) at (-2.5,-4.32)   {};
                            \node[vertex] (v6) at (2.5,-4.32)  {};
                            \node[vertex] (v7) at (10,0)  {};
                            \node[vertex] (v8) at (5,8.64)   {};
                            \node[vertex] (v9) at (-5,8.64)   {};
                            \node[vertex] (v10) at (-10,0)  {};
                            \node[vertex] (v11) at (-5,-8.64)   {};
                            \node[vertex] (v12) at (5,-8.64)  {};
                            \draw[edge] (v1)  -- (v2);
                            \draw[edge] (v2)  -- (v3);
                            \draw[edge] (v3)  -- (v4);
                            \draw[edge] (v4)  -- (v5);
                            \draw[selected edge] (v5)  -- (v6);
                            \draw[edge] (v6)  -- (v1);
                            \draw[selected edge] (v0)  -- (v1);
                            \draw[edge] (v0)  -- (v3);
                            \draw[selected edge] (v0)  -- (v5);
                            \draw[edge] (v7)  -- (v8);
                            \draw[edge] (v8)  -- (v9);
                            \draw[edge] (v9)  -- (v10);
                            \draw[edge] (v10)  -- (v11);
                            \draw[selected edge] (v11)  -- (v12);
                            \draw[edge] (v12)  -- (v7);
                            \draw[edge] (v2)  -- (v8);
                            \draw[edge] (v4)  -- (v10);
                            \draw[selected edge] (v6)  -- (v12);
                        \end{tikzpicture}}
                    \subfigure{
                        \begin{tikzpicture}[scale=0.09]
                            \tikzstyle{vertex}=[draw,circle,minimum size=4pt,inner sep=0pt]
                            \tikzstyle{edge} = [draw,line width=0.7pt,-,black]
                            \tikzstyle{selected edge} = [draw,line width=1.5pt,-,black]
                            \node[vertex] (v0) at (0,0)  {};
                            \node[vertex] (v1) at (5,0)  {};
                            \node[vertex] (v2) at (2.5,4.32)   {};
                            \node[vertex] (v3) at (-2.5,4.32)   {};
                            \node[vertex] (v4) at (-5,0)  {};
                            \node[vertex] (v5) at (-2.5,-4.32)   {};
                            \node[vertex] (v6) at (2.5,-4.32)  {};
                            \node[vertex] (v7) at (10,0)  {};
                            \node[vertex] (v8) at (5,8.64)   {};
                            \node[vertex] (v9) at (-5,8.64)   {};
                            \node[vertex] (v10) at (-10,0)  {};
                            \node[vertex] (v11) at (-5,-8.64)   {};
                            \node[vertex] (v12) at (5,-8.64)  {};
                            \draw[edge] (v1)  -- (v2);
                            \draw[edge] (v2)  -- (v3);
                            \draw[edge] (v3)  -- (v4);
                            \draw[edge] (v4)  -- (v5);
                            \draw[selected edge] (v5)  -- (v6);
                            \draw[edge] (v6)  -- (v1);
                            \draw[selected edge] (v0)  -- (v1);
                            \draw[edge] (v0)  -- (v3);
                            \draw[selected edge] (v0)  -- (v5);
                            \draw[edge] (v7)  -- (v8);
                            \draw[edge] (v8)  -- (v9);
                            \draw[edge] (v9)  -- (v10);
                            \draw[edge] (v10)  -- (v11);
                            \draw[edge] (v11)  -- (v12);
                            \draw[selected edge] (v12)  -- (v7);
                            \draw[edge] (v2)  -- (v8);
                            \draw[edge] (v4)  -- (v10);
                            \draw[selected edge] (v6)  -- (v12);
                        \end{tikzpicture}}
                        \subfigure{
                    \begin{tikzpicture}[scale=0.09]
                        \tikzstyle{vertex}=[draw,circle,minimum size=4pt,inner sep=0pt]
                            \tikzstyle{edge} = [draw,line width=0.7pt,-,black]
                            \tikzstyle{selected edge} = [draw,line width=1.5pt,-,black]
                            \node[vertex] (v0) at (0,0)  {};
                            \node[vertex] (v1) at (5,0)  {};
                            \node[vertex] (v2) at (2.5,4.32)   {};
                            \node[vertex] (v3) at (-2.5,4.32)   {};
                            \node[vertex] (v4) at (-5,0)  {};
                            \node[vertex] (v5) at (-2.5,-4.32)   {};
                            \node[vertex] (v6) at (2.5,-4.32)  {};
                            \node[vertex] (v7) at (10,0)  {};
                            \node[vertex] (v8) at (5,8.64)   {};
                            \node[vertex] (v9) at (-5,8.64)   {};
                            \node[vertex] (v10) at (-10,0)  {};
                            \node[vertex] (v11) at (-5,-8.64)   {};
                            \node[vertex] (v12) at (5,-8.64)  {};
                            \draw[selected edge] (v1)  -- (v2);
                            \draw[edge] (v2)  -- (v3);
                            \draw[edge] (v3)  -- (v4);
                            \draw[edge] (v4)  -- (v5);
                            \draw[selected edge] (v5)  -- (v6);
                            \draw[selected edge] (v6)  -- (v1);
                            \draw[edge] (v0)  -- (v1);
                            \draw[edge] (v0)  -- (v3);
                            \draw[selected edge] (v0)  -- (v5);
                            \draw[edge] (v7)  -- (v8);
                            \draw[edge] (v8)  -- (v9);
                            \draw[edge] (v9)  -- (v10);
                            \draw[edge] (v10)  -- (v11);
                            \draw[edge] (v11)  -- (v12);
                            \draw[edge] (v12)  -- (v7);
                            \draw[selected edge] (v2)  -- (v8);
                            \draw[edge] (v4)  -- (v10);
                            \draw[edge] (v6)  -- (v12);
                        \end{tikzpicture}}
                    \subfigure{
                        \begin{tikzpicture}[scale=0.09]
                            \tikzstyle{vertex}=[draw,circle,minimum size=4pt,inner sep=0pt]
                            \tikzstyle{edge} = [draw,line width=0.7pt,-,black]
                            \tikzstyle{selected edge} = [draw,line width=1.5pt,-,black]
                            \node[vertex] (v0) at (0,0)  {};
                            \node[vertex] (v1) at (5,0)  {};
                            \node[vertex] (v2) at (2.5,4.32)   {};
                            \node[vertex] (v3) at (-2.5,4.32)   {};
                            \node[vertex] (v4) at (-5,0)  {};
                            \node[vertex] (v5) at (-2.5,-4.32)   {};
                            \node[vertex] (v6) at (2.5,-4.32)  {};
                            \node[vertex] (v7) at (10,0)  {};
                            \node[vertex] (v8) at (5,8.64)   {};
                            \node[vertex] (v9) at (-5,8.64)   {};
                            \node[vertex] (v10) at (-10,0)  {};
                            \node[vertex] (v11) at (-5,-8.64)   {};
                            \node[vertex] (v12) at (5,-8.64)  {};
                            \draw[edge] (v1)  -- (v2);
                            \draw[edge] (v2)  -- (v3);
                            \draw[edge] (v3)  -- (v4);
                            \draw[selected edge] (v4)  -- (v5);
                            \draw[selected edge] (v5)  -- (v6);
                            \draw[selected edge] (v6)  -- (v1);
                            \draw[selected edge] (v0)  -- (v1);
                            \draw[edge] (v0)  -- (v3);
                            \draw[edge] (v0)  -- (v5);
                            \draw[edge] (v7)  -- (v8);
                            \draw[edge] (v8)  -- (v9);
                            \draw[edge] (v9)  -- (v10);
                            \draw[edge] (v10)  -- (v11);
                            \draw[edge] (v11)  -- (v12);
                            \draw[edge] (v12)  -- (v7);
                            \draw[edge] (v2)  -- (v8);
                            \draw[selected edge] (v4)  -- (v10);
                            \draw[edge] (v6)  -- (v12);
                        \end{tikzpicture}}
                    \caption{Examples of 5-length path having its head 3 edges in a square.}
                    \label{illustubf1}
                \end{figure*}
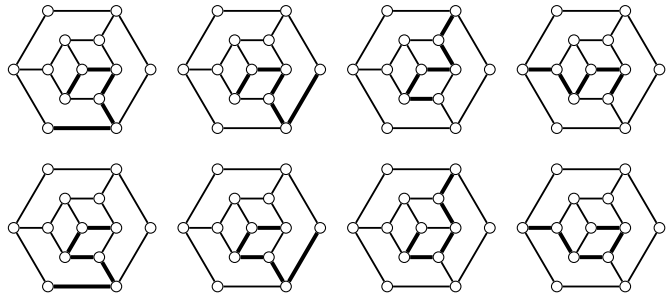
                \begin{figure*}[h]
                    \centering
                    \begin{tikzpicture}[scale=1]
                        \tikzstyle{vertex}=[draw,circle,minimum size=4pt,inner sep=0pt]
                        \tikzstyle{edge} = [draw,line width=0.7pt,-,black]
                        \tikzstyle{selected edge} = [draw,line width=1.5pt,-,black]
                        \node[vertex] (v1) at (0,0)  {};
                        \node[vertex] (v2) at (1,0)  {};
                        \node[vertex] (v3) at (2,0)  {};
                        \node[vertex] (v4) at (2,1)  {};
                        \node[vertex] (v5) at (1,1)  {};
                        \node[vertex] (v6) at (0,1)  {};
                        \draw[edge] (v1)  -- (v2);
                        \draw[selected edge] (v2)  -- (v3);
                        \draw[selected edge] (v3)  -- (v4);
                        \draw[edge] (v4)  -- (v5);
                        \draw[selected edge] (v5)  -- (v6);
                        \draw[selected edge] (v6)  -- (v1);
                        \draw[selected edge] (v2)  -- (v5);
                    \end{tikzpicture}
                    \caption{5-length path ``embedded" in a dual-square.}
                    \label{illustubf2}
                \end{figure*}
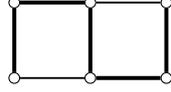
                \par
                First we consider the case  that the 5-length path chosen is in a 6-cycle. For different types of 6-cycle, hexagonal faces, boundary of square-cap with or without hexagon-layers and dual-square, we have $m-12$, $m-12$ and $m-11$ ways to add one edge, respectively. Hence, the number of subgraph $Q$ in this case is
                \begin{displaymath}
                    h\times6\times(m-12)+(t+1)\times6\times(m-12)+6\times6\times(m-11).
                \end{displaymath}
                \par
                Then we consider the case that some 3 consecutive edges of the 5-length path chosen are in a 4-cycle. If the inner 3 edges of the path is in a square (not including the circumstance where the whole path is in a dual-square), we have $6\times2$ such paths and $m-12$ ways to select the nomadic edge. If the 5-length path has its ``head"~3 edges in a square as shown in Figure~\ref{illustubf1}, then the number of  related subgraph $Q$ is $6\times8\times(m-12)$. If the path is embedded in a dual-square as shown in Figure~\ref{illustubf2}, the number of $Q$ is $6\times2\times(m-11)$.
                \par
                In the remaining cases, the number of $Q$ is
                \begin{displaymath}
                    [N(P)-h\times6-(t+1)\times6-6\times6-6\times2-6\times8-6\times2]\times(m-13).
                \end{displaymath}
                Thus, we have
                \begin{displaymath}\begin{split}
                 N(Q)&=6h\times(m-12)+(t+1)\times6\times(m-12)+36\times(m-11)\\
                    &{\ }+12\times(m-12)+48\times(m-12)+12\times(m-11)\\
                    &{\ }+[N(P)-6h-(t+1)\times6-36-12-48-12]\times(m-13).\\
                \end{split}\end{displaymath}
                Note that $h=3t$, and the proof is complete.
            \end{proof}
        \begin{lemma}\label{qlant}
            For a (4,6)-fullerene graph $G\notin\mathcal{T}$ other than the cube, we have $N(Q)=144h^2+318h+6y$, where  $y$ is the number of dual-squares of $G$.
        \end{lemma}
        \begin{proof}
            We assume that there are $x_k$ squares that have precisely $k$ neighbouring square faces, where $0\leqslant k \leqslant 2$.
            As in the proof of Lemma~\ref{qtub}, we consider different types of 5-length paths (subgraph $P$).

            First, if the 5-length path chosen is in a 6-cycle, then by Theorem~\ref{category} the number of resulted $Q$ is
            \begin{displaymath}
                h\times6\times(m - 12)+y\times6\times (m-11).
            \end{displaymath}
            Second, we consider whether some 3 consecutive edges of the chosen path are in a square. Also we should notice that the three edges in one square might be the ``head" three edges or the ``inner" three edges. We deal with a special case in the first place where the 5-length path chosen is embedded in a dual-square as shown in Figure~\ref{illustubf2}, the number of such resulted $Q$ is $2y\times(m-11)$. For the other cases,\\
            (a) if the square has no neighbouring squares, then the number of resulted $Q$ is
            \begin{displaymath}
                x_0\times4\times2\times2\times(m-12)+x_0\times4\times(m-12);
            \end{displaymath}
            (b) if the square has one neighbouring square, then the number of resulted $Q$ is
            \begin{displaymath}
                x_1\times12\times(m-12)+x_1\times3\times(m-12);
            \end{displaymath}
            (c) if the square has two neighbouring squares, then the number of resulted $Q$ is
            \begin{displaymath}
                x_2\times8\times(m-12)+x_2\times2\times(m-12).
            \end{displaymath}
            Third, in the remaining cases, the number of resulted $Q$ is
            \begin{displaymath}
                [N(P)-6h-6y-2y-20x_0-15x_1-10x_2]\times(m-13).
            \end{displaymath}
            Thus, we have
            \begin{displaymath}\begin{split}
                 N(Q)&=6h\times(m-12)+6y\times(m-11)\\
                    &{\ }+2y\times(m-11)+(20x_0+15x_1+10x_2)\times(m-12)\\
                    &{\ }+[N(P)-6h-6y-2y-(20x_0+15x_1+10x_2)]\times(m-13).\\
            \end{split}\end{displaymath}
            Observe that in a (4,6)-fullerene of lantern structure $x_0=0, x_1=4, x_2=2, y=4$, in a (4,6)-fullerene of dispersive structure $x_0=6-2y, x_1=2y, x_2=0$, and in a hexagonal prism, $x_0=x_1=0, x_2=6, y=6$. An elementary calculation completes our proof.
        \end{proof}
        \begin{lemma}\label{qcube}
            If $G$ is a cube, then $M(G,6)= 0$ and $N(Q)=96$.
        \end{lemma}
        \begin{proof}
            Theorem~\ref{res} claims that $M(G,5)=0$, which implies $M(G,6)=0$. Further, $N(Q)=96$ by~\eqref{mm6rec}.
        \end{proof}
        Now, we are in position to declare our enumeration result for 6-matchings.
        \begin{theorem}
            Let $G$ be a (4,6)-fullerene graph.\\
                (i) If $G$ is a cube, then $M(G,6)= 0$.\\
                (ii) If $G\in\mathcal{T}$, then
                    \begin{displaymath}
                        M(G,6)= \dfrac{81}{80}h^6-\dfrac{81}{80}h^5-\dfrac{99}{16}h^4+\dfrac{405}{16}h^3-\dfrac{1873}{40}h^2+\dfrac{1231}{30}h-9.
                    \end{displaymath}
                (iii) If $G\notin\mathcal{T}$  is different from  the cube and has $y$ dual-squares, then
                    \begin{displaymath}
                        M(G,6)= \dfrac{81}{80}h^6-\dfrac{81}{80}h^5-\dfrac{99}{16}h^4+\dfrac{405}{16}h^3-\dfrac{1873}{40}h^2+\dfrac{407}{10}h-16+y.
                    \end{displaymath}
        \end{theorem}
        \begin{proof}
            $M(G,6)$ in a cube has already been calculated in Lemma~\ref{qcube}. \eqref{mm6rec} together with Lemmas~\ref{qtub} and~\ref{qlant} imply our results for the other two cases.
        \end{proof}

        \begin{corollary}
            Let $G$ be a (4,6)-fullerene graph.\\
                (i) If $G$ is a cube, then $N(U)=N(S)=N(R)=0$.\\
                (ii) If $G\in\mathcal{T}$, then
                    \begin{displaymath}\begin{split}
                        N(U)&=108 h^3+216 h^2- 68h+ 54;\\
                        N(S)&=54h^4+ 108h^3- 138h^2+ 176h-54;\\
                        N(R)&=\dfrac{81}{4} h^5+\dfrac{27}{2} h^4-\dfrac{477}{4} h^3+\dfrac{471}{2} h^2-230 h+54.
                    \end{split}\end{displaymath}
                (iii) If $G\notin\mathcal{T}$ is different from  the cube and has $y$ dual-squares, then
                    \begin{displaymath}\begin{split}
                        N(U)&=108 h^3+216 h^2-66 h-6 y+96;\\
                        N(S)&=54 h^4+108 h^3-138 h^2+174 h+6 y-96;\\
                        N(R)&=\dfrac{81}{4} h^5+\dfrac{27}{2} h^4-\dfrac{477}{4} h^3+\dfrac{471}{2} h^2-228 h-6y+96.
                    \end{split}\end{displaymath}
        \end{corollary}
        \begin{proof}
            This corollary follows from Corollary~\ref{coroEX}, and Lemmas~\ref{qtub},~\ref{qlant} and~\ref{qcube} immediately.
        \end{proof}

\end{document}